\documentclass[11pt]{amsart}

\usepackage[T1]{fontenc}
\usepackage{amssymb,mathrsfs}
\usepackage{palatino,microtype}
\usepackage{graphicx}
\usepackage{xcolor}
\usepackage{tikz}
\usetikzlibrary{arrows.meta,patterns,calc}
\usepackage{pgfplots}
\usepgfplotslibrary{groupplots}
\pgfplotsset{compat=1.18}
\usepackage[colorlinks=true,linkcolor=blue,citecolor=blue,urlcolor=blue]{hyperref}

\numberwithin{equation}{section}
\allowdisplaybreaks
\theoremstyle{plain}
\newtheorem{theorem}{Theorem}[section]
\newtheorem{proposition}[theorem]{Proposition}
\newtheorem{lemma}[theorem]{Lemma}
\newtheorem{corollary}[theorem]{Corollary}

\theoremstyle{remark}
\newtheorem{remark}[theorem]{Remark}

\begin{document}

\title[Triangle averages and maximal operators]{Asymmetric $L^p$-improving and maximal estimates for the triangle averaging operator}

\author{Alex Iosevich}
\thanks{The first author was supported in part by NSF grant DMS-2506858 and a Simons Fellowship in Mathematics in 2026.}
\address{Department of Mathematics, University of Rochester, Rochester, NY 14627}
\email{alex.iosevich@rochester.edu}

\author{Eyvindur Ari Palsson}
\address{Department of Mathematics, Virginia Tech, Blacksburg, VA 24061}
\email{palsson@vt.edu}

\author{Minh-Quy Pham}
\address{Department of Mathematics, Virginia Tech, Blacksburg, VA 24061}
\email{mqpham@vt.edu}

\date{\today}

\begin{abstract}
Let $d\ge2$ and $t>0$. We study the bilinear triangle averaging operator
$$\mathcal T_t(f,g)(x)=\int_{\mathcal I}f(x-ty)g(x-tz) d\sigma(y,z),$$
where $\sigma$ is the normalized surface measure on the manifold $\mathcal I=\{(y,z)\in\mathbb R^d\times\mathbb R^d:|y|=|z|=|y-z|=1\}$ of equilateral triangles of side length one with a vertex at the origin. We prove the strong-type asymmetric $L^p$-improving estimate
$$\mathcal T_1:L^{\frac{d+1}{d-1}}(\mathbb R^d)\times L^{\frac{d^2-1}{d^2-3d+4}}(\mathbb R^d)\longrightarrow L^{\frac{d+1}{2}}(\mathbb R^d)$$
in every dimension $d\ge4$. This enlarges the range of exponents $(p,q,r)$ for which the triangle averaging operator is known to map $L^p(\mathbb R^d)\times L^q(\mathbb R^d)$ to $L^r(\mathbb R^d)$, and symmetrization produces new symmetric estimates in every such dimension as well. The proof combines Sobolev smoothing estimates, obtained through a microlocal decomposition of the triangle incidence relation, with complex interpolation in an analytic family of averages over triangles of varying shape that contains the triangle averaging operator.

We also establish maximal estimates for triangle averages. Let $E\subset[1,2]$ be a nonempty set of upper Minkowski dimension $\beta$, and let $\mathcal T_E(f,g)=\sup_{t\in E}|\mathcal T_t(f,g)|$. We prove that
$$\|\mathcal T_E(f,g)\|_2\lesssim\|f\|_{H^{-s_1}}\|g\|_{H^{-s_2}}$$
whenever $s_1,s_2\ge0$, $s_1+s_2<\frac{d-3-\beta}{2}$ and $d>3+\beta$. This improves the previously known Sobolev smoothing estimate by half a derivative and lowers the dimension requirement. We also obtain asymmetric localized maximal estimates of the form $\mathcal T_E:L^p(\mathbb R^d)\times L^q(\mathbb R^d)\to L^r(\mathbb R^d)$ in an enlarged range of exponents.
\end{abstract}

\maketitle

\section{Introduction}

Let $d\ge2$, and let
$$\mathcal I=\{(y,z)\in\mathbb R^d\times\mathbb R^d:|y|=|z|=|y-z|=1\}$$
be the submanifold of $\mathbb R^{2d}$ consisting of all pairs $(y,z)$ for which $\{0,y,z\}$ is the set of vertices of an equilateral triangle of side length one. For $t>0$ we define the bilinear triangle averaging operator by
\begin{equation}\label{eq_tri_op}
\mathcal T_t(f,g)(x)=\int_{\mathcal I}f(x-ty)g(x-tz) d\sigma(y,z),
\end{equation}
where $\sigma$ denotes the normalized surface measure on $\mathcal I$. Thus $\mathcal T_t$ averages the product $f(x-ty)g(x-tz)$ over all equilateral triangles of side length $t$ with one vertex at $x$, as in Figure~\ref{fig:triangle-config}. Any two pairs of unit vectors with the same inner product differ by an orthogonal transformation, so the orthogonal group $O(d)$ acts transitively on $\mathcal I$, and $\sigma$ is the unique $O(d)$-invariant probability measure on $\mathcal I$. This description of $\sigma$ will be used several times.

\begin{figure}[htbp]
\centering
\begin{tikzpicture}[scale=1.05]
\coordinate (X) at (0,0);
\draw[gray!70,dashed] (X) circle (2.2);
\coordinate (Y) at ({2.2*cos(200)},{2.2*sin(200)});
\coordinate (Z) at ({2.2*cos(260)},{2.2*sin(260)});
\coordinate (Y2) at ({2.2*cos(40)},{2.2*sin(40)});
\coordinate (Z2) at ({2.2*cos(100)},{2.2*sin(100)});
\fill[blue!10] (X) -- (Y) -- (Z) -- cycle;
\draw[thick,blue!60!black] (X) -- (Y) -- (Z) -- cycle;
\fill[green!10] (X) -- (Y2) -- (Z2) -- cycle;
\draw[thick,green!45!black,densely dashed] (X) -- (Y2) -- (Z2) -- cycle;
\fill (X) circle (1.6pt);
\fill (Y) circle (1.6pt);
\fill (Z) circle (1.6pt);
\fill (Y2) circle (1.6pt);
\fill (Z2) circle (1.6pt);
\node[anchor=north west] at (0.08,-0.05) {$x$};
\node[anchor=east] at ($(Y)+(-0.08,0)$) {$x-ty$};
\node[anchor=north] at ($(Z)+(0,-0.08)$) {$x-tz$};
\node[anchor=west] at ($(Y2)+(0.08,0)$) {$x-ty'$};
\node[anchor=south] at ($(Z2)+(0,0.08)$) {$x-tz'$};
\node[blue!60!black] at ($(Y)!0.5!(Z)+(-0.25,-0.2)$) {$t$};
\node[gray!70!black,anchor=west] at (2.0,-1.5) {radius $t$};
\end{tikzpicture}
\caption{The operator $\mathcal T_t$ averages $f(x-ty)g(x-tz)$ over all equilateral triangles of side length $t$ with one vertex at $x$. The vertex $x-ty$ ranges over the sphere of radius $t$ about $x$ (dashed), and once it is chosen, the vertex $x-tz$ ranges over a sphere of dimension $d-2$. Two such triangles are shown.}
\label{fig:triangle-config}
\end{figure}
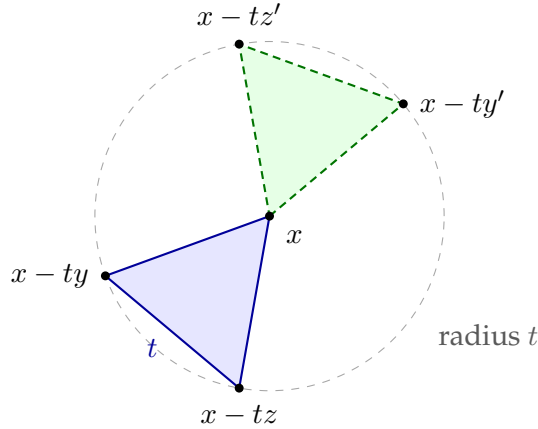

Averaging operators over curved manifolds have been studied extensively in harmonic analysis. The classical example is the spherical averaging operator
$$A_tf(x)=\int_{S^{d-1}}f(x-ty) d\sigma_{S^{d-1}}(y),$$
where $\sigma_{S^{d-1}}$ is the normalized surface measure on the unit sphere. Its $L^p$-improving properties were established by Strichartz \cite{Strichartz1970} (see also Littman \cite{Littman1973}), and the associated maximal operator was studied by Stein \cite{Stein1976} for $d\ge3$ and by Bourgain \cite{Bourgain1986} for $d=2$.

Bilinear spherical averages and the related maximal operators have attracted considerable attention in recent years \cite{GebaGreenleafIosevichPalssonSawyer2013,BarrionuevoGrafakosHeHonzikOliveira2018,JeongLee2020}. Subsequent developments include sparse domination and weighted estimates \cite{RoncalShrivastavaShuin2021,BorgesFosterOuPipherZhou2023,PalssonSovine2026}, lacunary and fractal-dilation maximal estimates \cite{BorgesFoster2024,BorgesFosterOu2025,ChoudharyShenShrivastava2026}, and endpoint estimates \cite{BhojakChoudharyShrivastavaShuin2024,Vergara2026}, and the survey \cite{BorgesSurvey2026} offers a broader account of the area.

Triangle averaging operators were studied in \cite{GreenleafIosevich2012,PalssonSovine2020,IosevichPalssonSovine2022}. Maximal estimates for simplices were obtained by Cook, Lyall and Magyar \cite{CookLyallMagyar2021}, with more recent developments by Borges and Foster \cite{BorgesFoster2024}, by Borges, Foster and Ou \cite{BorgesFosterOu2025}, and by Palsson and Sovine \cite{PalssonSovine2026}. These operators are closely connected with Falconer-type configuration problems \cite{GreenleafIosevich2012,GrafakosGreenleafIosevichPalsson2015}.

\subsection{Notation}
We write $\mathbb N_0=\{0,1,2,\ldots\}$, we use $\mathcal S(\mathbb R^n)$ for the Schwartz space, $\delta_0$ for the Dirac mass at the origin of the ambient Euclidean space and $\Gamma$ for the Gamma function, and we use the convention $\frac{1}{\infty}=0$. The Fourier transform is normalized by
$$\widehat f(\xi)=\int_{\mathbb R^n}e^{-2\pi ix\cdot\xi}f(x) dx,$$
and for $s\in\mathbb R$ the inhomogeneous Sobolev space $H^s(\mathbb R^n)$ carries the norm
$$\|f\|_{H^s(\mathbb R^n)}^2=\int_{\mathbb R^n}\langle\xi\rangle^{2s}|\widehat f(\xi)|^2 d\xi,$$
where $\langle\xi\rangle=(1+|\xi|^2)^{\frac12}$. For $s<0$ the elements of $H^s(\mathbb R^n)$ are understood as tempered distributions.

We write $\|f\|_p=\|f\|_{L^p(\mathbb R^d)}$, and for $0<r<1$ the symbol $\|\cdot\|_r$ denotes the usual $L^r$ quasi-norm. The H\"older conjugate of $1<p<\infty$ is $p'=\frac{p}{p-1}$, and $1'=\infty$, $\infty'=1$. The notation $X\lesssim Y$ means that $X\le CY$ for a constant $C$ independent of the functions and of the frequency indices under consideration, and the dependence of $C$ on geometric parameters and covering constants is indicated in the relevant statements. We write $X\approx Y$ when $X\lesssim Y$ and $Y\lesssim X$.

\subsection{Fixed-radius \texorpdfstring{$L^p$}{Lp}-improving estimates}

We first consider $L^p$-improving properties of the fixed-radius operator $\mathcal T_1$. Palsson and Sovine \cite{PalssonSovine2020} obtained H\"older type $L^p\times L^q\to L^r$ bounds for triangle averages in dimensions $d\ge5$, including quasi-Banach target spaces. Sharp $L^p$-improving estimates for the triangle averaging operator in the Banach range were obtained in the plane by Greenleaf, Iosevich, Krause and Liu \cite{GreenleafIosevichKrauseLiu2023}, and Iosevich, Palsson and Sovine \cite{IosevichPalssonSovine2022} subsequently extended these results to higher dimensions, establishing $L^p$-improving estimates for all $d\ge2$ that include
$$\mathcal T_1:L^{\frac{d+1}{d}}(\mathbb R^d)\times L^{\frac{d+1}{d}}(\mathbb R^d)\longrightarrow L^r(\mathbb R^d)$$
for $\frac{d+1}{2d}\le r\le1$. Unlike the planar results, these estimates are not expected to be sharp in general.

It is convenient to record such estimates through reciprocal exponents, and we say that $\mathcal T_1$ is bounded at the point $(\frac1p,\frac1q,\frac1r)$ if $\|\mathcal T_1(f,g)\|_r\lesssim\|f\|_p\|g\|_q$. To measure what our results add, we collect the estimates available from earlier work into a comparison region. Since $\mathcal T_1$ is an average with respect to a probability measure, it is bounded at $(0,0,0)$, $(1,0,1)$ and $(0,1,1)$. The pointwise bound $|\mathcal T_1(f,g)|\le\|g\|_\infty A_1|f|$, together with the spherical estimate $A_1:L^{\frac{d+1}{d}}\to L^{d+1}$ of Strichartz \cite{Strichartz1970} and Littman \cite{Littman1973}, shows that $\mathcal T_1$ is bounded at $(\alpha,0,b)$ and, by symmetry, at $(0,\alpha,b)$, where $\alpha=\frac{d}{d+1}$ and $b=\frac{1}{d+1}$, while the estimates of Iosevich, Palsson and Sovine quoted above correspond to the points $(\alpha,\alpha,1)$ and $(\alpha,\alpha,2\alpha)$. This leads to the polytope
$$\mathcal B_d=\operatorname{conv}\{(0,0,0),(1,0,1),(0,1,1),(\alpha,0,b),(0,\alpha,b),(\alpha,\alpha,1),(\alpha,\alpha,2\alpha)\}.$$

A second source of estimates is the product bound of Cook, Lyall and Magyar \cite{CookLyallMagyar2021}. Let $m=\lfloor\frac d2\rfloor$ and $\vartheta=\frac{m-1}{m}$. In the form recorded in \cite[Theorem~2.1]{IosevichPalssonSovine2022}, their estimate reads
\begin{equation}\label{eq:CLM-product}
|\mathcal T_1(f,g)|\lesssim(A_1|f|^{\frac{1}{\vartheta}})^{\vartheta}(A_1|g|^{\frac{1}{\vartheta}})^{\vartheta}.
\end{equation}
The dimensional restriction in that theorem is $d\ge km$, where $k,m\ge2$ are integers and $k=2$ for triangles, so our choice of $m$ is admissible for every $d\ge4$. The operator $A_1$ maps $L^{\frac{1}{a_1}}$ to $L^{\frac{1}{b_1}}$ for each pair $(a_1,b_1)$ in the set $\mathcal V_d=\{(0,0),(1,1),(\alpha,b)\}$, and therefore \eqref{eq:CLM-product} and H\"older's inequality show that $\mathcal T_1$ is bounded at every point of
\begin{equation}\label{eq:Gregion}
\mathcal G_d=\{(\vartheta a_1,\vartheta a_2,\vartheta(b_1+b_2)):(a_1,b_1),(a_2,b_2)\in\mathcal V_d\}.
\end{equation}

Duality and symmetry turn these estimates into further ones. If $\frac1r\le1$, then $\mathcal T_1$ is bounded at $(\frac1p,\frac1q,\frac1r)$ if and only if the trilinear form
$$\Lambda(f_1,f_2,f_3)=\int_{\mathbb R^d}\mathcal T_1(f_1,f_2)(x)f_3(x) dx$$
satisfies $|\Lambda(f_1,f_2,f_3)|\lesssim\|f_1\|_p\|f_2\|_q\|f_3\|_{r'}$, and we call $(\frac1p,\frac1q,1-\frac1r)$ the trilinear coordinates of the point. The form $\Lambda$ integrates $f_1(x-y)f_2(x-z)f_3(x)$ over all equilateral triangles $\{x-y,x-z,x\}$ of side length one in $\mathbb R^d$, and the natural measure on these triangles does not distinguish between their vertices, so $\Lambda$ is invariant under all permutations of $f_1,f_2,f_3$, as we verify in Lemma~\ref{lem:symmetry}. Permuting the trilinear coordinates of a point at which $\mathcal T_1$ is bounded with a Banach target space therefore produces another such point. Let $\mathcal C_d$ be the union of $\mathcal G_d$ with all triples obtained in this way from the points of $\mathcal G_d$ whose third coordinate is at most one, and define the comparison region
$$\widehat{\mathcal B}_d=\operatorname{conv}(\mathcal B_d\cup\mathcal C_d).$$
Lemma~\ref{lem:positive-interpolation} below allows us to interpolate between bounds for positive averages, including those with quasi-Banach targets, so $\mathcal T_1$ is bounded at every point of $\widehat{\mathcal B}_d$. We use this region as the benchmark for our results.

Set $U_d=\frac{d-1}{d+1}$ and $V_d=\frac{d^2-3d+4}{d^2-1}$. The following theorem provides a strong-type estimate outside $\widehat{\mathcal B}_d$ in every dimension $d\ge4$.

\begin{theorem}\label{thm:asymmetric}
Let $d\ge4$. Then the triangle averaging operator satisfies the strong-type estimate
\begin{equation}\label{eq:asymmetric}
\mathcal T_1:L^{\frac{d+1}{d-1}}(\mathbb R^d)\times L^{\frac{d^2-1}{d^2-3d+4}}(\mathbb R^d)\longrightarrow L^{\frac{d+1}{2}}(\mathbb R^d).
\end{equation}
Moreover, $\mathcal T_1$ is bounded from $L^p(\mathbb R^d)\times L^q(\mathbb R^d)$ to $L^r(\mathbb R^d)$ whenever $(\frac1p,\frac1q,\frac1r)\in\mathcal Q_d$, where
$$\mathcal Q_d=\operatorname{conv}(\widehat{\mathcal B}_d\cup\{(U_d,V_d,1-U_d),(V_d,U_d,1-U_d),(U_d,U_d,1-V_d)\}).$$
\end{theorem}

The estimate \eqref{eq:asymmetric} says that $\mathcal T_1$ is bounded at $(U_d,V_d,1-U_d)$, a point whose trilinear coordinates are $(U_d,V_d,U_d)$. The other two new points, with trilinear coordinates $(V_d,U_d,U_d)$ and $(U_d,U_d,V_d)$, come from \eqref{eq:asymmetric} through the symmetry of the triangle, and averaging the three permutations leads to the following symmetric estimate.

\begin{corollary}\label{cor:symmetric-fixed}
Let $d\ge4$. Then
$$\mathcal T_1:L^{\frac{3(d^2-1)}{3d^2-7d+6}}(\mathbb R^d)\times L^{\frac{3(d^2-1)}{3d^2-7d+6}}(\mathbb R^d)\longrightarrow L^{\frac{3(d^2-1)}{7d-9}}(\mathbb R^d).$$
In particular,
$$\mathcal T_1:L^{\frac{d+3}{d}}(\mathbb R^d)\times L^{\frac{d+3}{d}}(\mathbb R^d)\longrightarrow L^{\frac{d+3}{3}}(\mathbb R^d).$$
\end{corollary}

To see the improvement over earlier results concretely, consider $d=4$ and $d=5$. In these dimensions $\vartheta=\frac12$, and the proof of Theorem~\ref{thm:fixed-region} below shows that every point of $\widehat{\mathcal B}_d$ satisfies
$$1+\frac1p+\frac1q-\frac1r\le\frac{2d}{d+1}.$$
On the section $\frac1r=\frac{2}{d+1}$ this means that $\frac1p+\frac1q\le1$. The exponents of Theorem~\ref{thm:asymmetric} lie on the same section, but they satisfy $\frac1p+\frac1q=\frac{17}{15}$ when $d=4$ and $\frac1p+\frac1q=\frac54$ when $d=5$. The new estimates therefore lie strictly outside the comparison region, as Figure~\ref{fig:intro-fixed-compare} shows, and Theorem~\ref{thm:fixed-region} establishes the same conclusion in every dimension $d\ge4$.

\begin{figure}[htbp]
\centering
\begin{tikzpicture}
\begin{groupplot}[
    group style={group size=2 by 1, horizontal sep=2.2cm},
    width=0.42\textwidth,
    height=0.42\textwidth,
    xmin=-0.03, xmax=1.03,
    ymin=-0.03, ymax=1.03,
    axis lines=left,
    xlabel={$\frac{1}{p}$},
    ylabel={$\frac{1}{q}$},
    ylabel style={rotate=-90},
    xtick={0,0.5,1},
    ytick={0,0.5,1},
    grid=major,
    grid style={gray!25},
    enlargelimits=false,
    clip=false,
]

\nextgroupplot[title={$d=4$, $r=\frac{5}{2}$}]

\addplot[draw=red!75!black, fill=white!12, very thick, dashed]
coordinates {(0,{2/5}) ({2/5},0) ({17/20},0) ({17/25},{17/25}) (0,{17/20})} -- cycle;

\addplot[draw=blue!70!black, fill=blue!12, thick, dashed]
coordinates {(0,{17/20}) (0,{2/5}) ({2/5},0) ({17/20},0) ({4/5},{1/5}) ({1/5},{4/5})} -- cycle;

\addplot[draw=green!50!black, fill=green!20, pattern=north east lines, pattern color=green!50!black, thick]
coordinates {({4/5},{1/5}) ({3/5},{8/15}) ({8/15},{3/5}) ({1/5},{4/5})} -- cycle;

\addplot[draw=green!60!black, very thick]
coordinates {(0,{17/20}) (0,{2/5}) ({2/5},0) ({17/20},0) ({4/5},{1/5}) ({3/5},{8/15}) ({8/15},{3/5}) ({1/5},{4/5})} -- cycle;

\addplot[densely dotted, gray!60, domain=0:1] {x};

\addplot[only marks, mark=diamond*, mark size=2.8pt, color=green!50!black]
coordinates {({3/5},{8/15})};

\draw[green!50!black] (axis cs:{3/5},{8/15}) -- (axis cs:0.75,0.62);
\node[anchor=west, text=green!50!black] at (axis cs:0.75,0.62) {$\left(\frac{3}{5},\frac{8}{15}\right)$};

\nextgroupplot[title={$d=5$, $r=3$}]

\addplot[draw=red!75!black, fill=white!12, very thick, dashed]
coordinates {(0,{1/3}) ({1/3},0) ({13/15},0) ({13/18},{13/18}) (0,{13/15})} -- cycle;

\addplot[draw=blue!70!black, fill=blue!12, thick, dashed]
coordinates {(0,{13/15}) (0,{1/3}) ({1/3},0) ({13/15},0) ({5/6},{1/6}) ({1/6},{5/6})} -- cycle;

\addplot[draw=green!50!black, fill=green!20, pattern=north east lines, pattern color=green!50!black, thick]
coordinates {({5/6},{1/6}) ({2/3},{7/12}) ({7/12},{2/3}) ({1/6},{5/6})} -- cycle;

\addplot[draw=green!60!black, very thick]
coordinates {(0,{13/15}) (0,{1/3}) ({1/3},0) ({13/15},0) ({5/6},{1/6}) ({2/3},{7/12}) ({7/12},{2/3}) ({1/6},{5/6})} -- cycle;

\addplot[densely dotted, gray!60, domain=0:1] {x};

\addplot[only marks, mark=diamond*, mark size=2.8pt, color=green!50!black]
coordinates {({2/3},{7/12})};

\draw[green!50!black] (axis cs:{2/3},{7/12}) -- (axis cs:0.79,0.66);
\node[anchor=west, text=green!50!black] at (axis cs:0.79,0.66) {$\left(\frac{2}{3},\frac{7}{12}\right)$};

\end{groupplot}
\end{tikzpicture}
\caption{Sections of the fixed-radius boundedness regions at $\frac1r=\frac{2}{d+1}$ in dimensions $d=4$ and $d=5$, drawn in the $(\frac1p,\frac1q)$-plane. The blue polygons are the sections of the comparison region $\widehat{\mathcal B}_d$. The green hatched regions show the additional boundedness obtained from Theorem~\ref{thm:asymmetric} by symmetry and interpolation with the comparison estimates, so that the blue and green regions together form the corresponding sections of $\mathcal Q_d$. The diamond marks the estimate \eqref{eq:asymmetric}. The red dashed polygons are the outer regions determined by the necessary conditions of Proposition~\ref{prop:triangle-necessary}.}
\label{fig:intro-fixed-compare}
\end{figure}
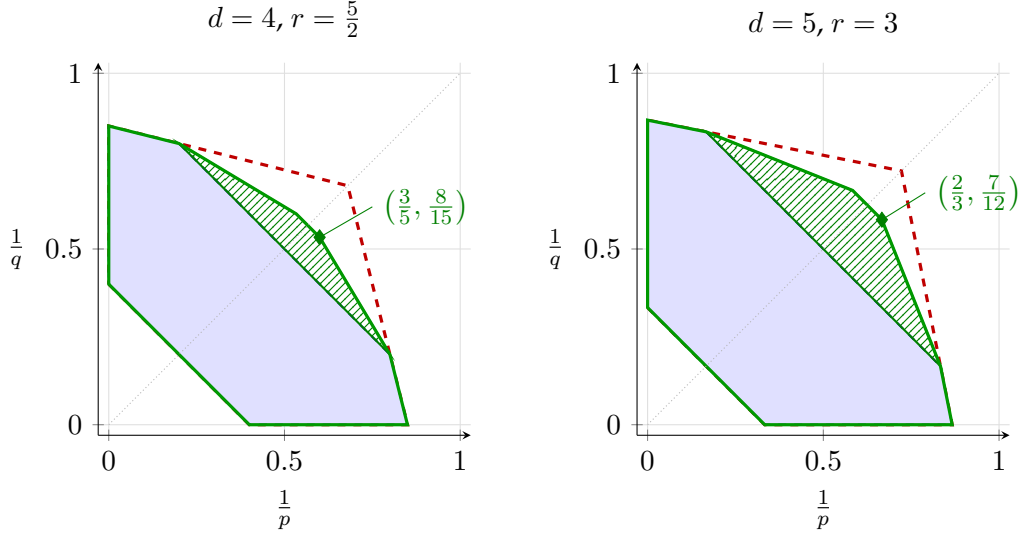

\subsection{Sobolev smoothing estimates for localized maximal triangle averages}

We next consider maximal operators associated with triangle averages. For a nonempty set $E\subset[1,2]$ define
$$\mathcal T_E(f,g)(x)=\sup_{t\in E}|\mathcal T_t(f,g)(x)|.$$
The map $\mathcal T_E$ is sublinear in each input separately, and we call it localized because the dilation parameter is restricted to $E\subset[1,2]$. Its behavior depends on the size of $E$, which we measure by the upper Minkowski dimension. Writing $N(E,\delta)$ for the least number of intervals of length $\delta$ needed to cover $E$, we set
$$\beta=\overline{\dim}_{\mathrm M}E=\limsup_{\delta\downarrow0}\frac{\log N(E,\delta)}{\log\frac{1}{\delta}}.$$
Thus $0\le\beta\le1$, finite sets have $\beta=0$, intervals have $\beta=1$, and for every $\varepsilon>0$ there is a constant $C_{E,\varepsilon}$ such that
\begin{equation}\label{eq:minkowski-cover}
N(E,\delta)\le C_{E,\varepsilon}\delta^{-\beta-\varepsilon}
\end{equation}
for all $0<\delta\le1$. The constants in our maximal estimates may depend on these covering constants.

Palsson and Sovine \cite{PalssonSovine2020} obtained initial bounds for the full triangle maximal operator using spherical maximal estimates and interpolation. Cook, Lyall and Magyar \cite{CookLyallMagyar2021} established maximal bounds for nondegenerate simplices through pointwise estimates in terms of spherical maximal functions, and Borges and Foster \cite{BorgesFoster2024} subsequently studied lacunary triangle maximal operators. More recently, Borges, Foster and Ou \cite{BorgesFosterOu2025} investigated maximal triangle averages over fractal sets of radii and established the Sobolev smoothing estimate
$$\|\mathcal T_E(f,g)\|_{L^2(\mathbb R^d)}\lesssim\|f\|_{H^{-s_1}(\mathbb R^d)}\|g\|_{H^{-s_2}(\mathbb R^d)}$$
under the conditions $d>4+\beta$, $s_1,s_2\ge0$ and $s_1+s_2<\frac{d-4-\beta}{2}$, where $\beta$ is the upper Minkowski dimension of $E$. Sparse domination for maximal triangle averaging operators was developed further by Palsson and Sovine \cite{PalssonSovine2026}. Our first maximal result improves the estimate of Borges, Foster and Ou by half a derivative and lowers the dimension requirement from $d>4+\beta$ to $d>3+\beta$.

\begin{theorem}\label{thm:main_Sobolev_fractal_dilation}
Let $E\subset[1,2]$ be nonempty, with upper Minkowski dimension $\beta\in[0,1]$, and suppose that $d>3+\beta$. For every $s_1,s_2\ge0$ satisfying
$$s_1+s_2<\frac{d-3-\beta}{2}$$
we have
$$\|\mathcal T_E(f,g)\|_{L^2(\mathbb R^d)}\lesssim\|f\|_{H^{-s_1}(\mathbb R^d)}\|g\|_{H^{-s_2}(\mathbb R^d)}.$$
\end{theorem}

\subsection{\texorpdfstring{$L^p$}{Lp}-improving estimates for localized maximal triangle averages}

Combining the Sobolev smoothing estimate with the asymmetric fixed-radius bound, we obtain $L^p$-improving estimates for localized maximal triangle averages. For $0\le\theta\le1$ define
\begin{equation}\label{eq:asymmaxexponents}
\begin{aligned}
\frac{1}{p_\theta}&=\frac{\theta}{2}+(1-\theta)U_d,\\
\frac{1}{q_\theta}&=\frac{\theta}{2}+(1-\theta)V_d,\\
\frac{1}{r_\theta}&=\frac{\theta}{2}+(1-\theta)\frac{2}{d+1}.
\end{aligned}
\end{equation}
As $\theta$ increases from $0$ to $1$, the point $(\frac{1}{p_\theta},\frac{1}{q_\theta},\frac{1}{r_\theta})$ moves along the segment joining $(U_d,V_d,1-U_d)$, the point of Theorem~\ref{thm:asymmetric}, to $(\frac12,\frac12,\frac12)$, the point of the $L^2$ Sobolev estimate. For $d\ge4$ these exponents satisfy $1<p_\theta,q_\theta\le2\le r_\theta<\infty$, and in general the two input exponents are different.

\begin{theorem}\label{thm:mainasymmax}
Let $d\ge4$, and let $E\subset[1,2]$ be nonempty with $\overline{\dim}_{\mathrm M}E=\beta$. Suppose that $d>3+\beta$, and let
$$\theta_0=\frac{4\beta}{(d-3)(d+1-\beta)}.$$
If $\theta_0<\theta\le1$, then
\begin{equation}\label{eq:mainasymmax}
\mathcal T_E:L^{p_\theta}(\mathbb R^d)\times L^{q_\theta}(\mathbb R^d)\longrightarrow L^{r_\theta}(\mathbb R^d),
\end{equation}
where the exponents are defined in \eqref{eq:asymmaxexponents}. Moreover, define the reciprocal exponent triples $P_\theta=(\frac{1}{p_\theta},\frac{1}{q_\theta},\frac{1}{r_\theta})$ and $\widetilde P_\theta=(\frac{1}{q_\theta},\frac{1}{p_\theta},\frac{1}{r_\theta})$, and let
$$\mathcal R_{d,\beta}=\operatorname{conv}(\{(0,0,0)\}\cup\{P_\theta,\widetilde P_\theta:\theta_0<\theta\le1\}).$$
Then $\mathcal T_E:L^p(\mathbb R^d)\times L^q(\mathbb R^d)\to L^r(\mathbb R^d)$ whenever $(\frac1p,\frac1q,\frac1r)\in\mathcal R_{d,\beta}$.
\end{theorem}

Theorem~\ref{thm:mainasymmax} establishes asymmetric $L^p$-improving estimates for localized maximal triangle averages over sets of radii of prescribed upper Minkowski dimension. The dimension range $d>3+\beta$ comes from our Sobolev smoothing theorem, whereas the corresponding result of Borges, Foster and Ou \cite{BorgesFosterOu2025} requires $d>4+\beta$. Two examples illustrate the gain. When $E=[1,2]$, so that $\beta=1$, Theorem~\ref{thm:mainasymmax} applies already in dimension $d=5$, while the corresponding result in \cite{BorgesFosterOu2025} requires $d>5$. Taking $\theta=\frac12$, which exceeds $\theta_0=\frac25$, we get
$$\mathcal T_{[1,2]}:L^{\frac{12}{7}}(\mathbb R^5)\times L^{\frac{24}{13}}(\mathbb R^5)\longrightarrow L^{\frac{12}{5}}(\mathbb R^5),$$
and interpolation with the estimate obtained by interchanging the two inputs leads to
$$\mathcal T_{[1,2]}:L^{\frac{16}{9}}(\mathbb R^5)\times L^{\frac{16}{9}}(\mathbb R^5)\longrightarrow L^{\frac{12}{5}}(\mathbb R^5).$$
If $\beta<1$, the theorem applies in dimension $d=4$, a case not covered by \cite{BorgesFosterOu2025} at all. For example, when $\beta=0$ and $\theta=\frac14$,
$$\mathcal T_E:L^{\frac{40}{23}}(\mathbb R^4)\times L^{\frac{40}{21}}(\mathbb R^4)\longrightarrow L^{\frac{40}{17}}(\mathbb R^4).$$

The asymmetric estimates of Theorem~\ref{thm:mainasymmax} form part of a larger sufficient region, which also draws on the one-sided spherical maximal estimates of Anderson, Hughes, Roos and Seeger \cite{AndersonHughesRoosSeeger2021}, the strong type endpoint for the spherical maximal function due to Lee and Zhao \cite{LeeZhao2026}, and the product estimate \eqref{eq:CLM-product} of Cook, Lyall and Magyar. To describe it, let
$$M_E^{\mathrm{sph}}h(x)=\sup_{t\in E}|A_th(x)|$$
be the spherical maximal operator over $E$, and for $0\le\beta\le1$ consider the points
$$\begin{aligned}
S_1&=(0,0), & S_2&=\left(\frac{d-1}{d-1+\beta},\frac{d-1}{d-1+\beta}\right),\\
S_3&=\left(\frac{d-\beta}{d-\beta+1},\frac{1}{d-\beta+1}\right), & S_4&=\left(\frac{d(d-1)}{d^2+1},\frac{d-1}{d^2+1}\right)
\end{aligned}$$
in the plane of reciprocal exponents $(a,b)$ of a linear operator from $L^{\frac1a}$ to $L^{\frac1b}$. Let $\mathscr S_{d,\beta}$ be the convex hull of the union of three sets, namely the interior of the quadrilateral $\operatorname{conv}\{S_1,S_2,S_3,S_4\}$, the diagonal segment $\{(a,a):0\le a<\frac{d-1}{d-1+\beta}\}$ and the single point $S_4$. Here and below $\operatorname{conv}$ denotes the convex hull itself, not its closure. In Subsection~\ref{sec:combined} we show that $M_E^{\mathrm{sph}}$ maps $L^{\frac1a}$ to $L^{\frac1b}$ for every $(a,b)\in\mathscr S_{d,\beta}$.

Writing points of the space of reciprocal exponents as $(x,y,z)$, we now introduce four families of such points. The first one is
$$\mathcal E_{d,\beta}=\left\{\left(x,y,\frac12\right):x\ge\frac12,y\ge\frac12,x+y<\frac32-\frac{3+\beta}{2d}\right\}.$$
The second, $\mathcal J^{\mathrm{all}}_{d,\beta}$, consists of the points $\theta(\frac12,\frac12,\frac12)+(1-\theta)(x_0,y_0,z_0)$ with $(x_0,y_0,z_0)\in\mathcal Q_d$, $0<x_0<1$, $0<y_0<1$, $0<z_0\le 1$, $0<\theta\le1$ and
\begin{equation}\label{eq:full-transfer-region}
\theta\frac{d-3}{2}>\beta\left(\frac{\theta}{2}+(1-\theta)z_0\right).
\end{equation}
The remaining two families are
$$\mathcal L_{d,\beta}=\{(a,0,b),(0,a,b):(a,b)\in\mathscr S_{d,\beta}\}$$
and
$$\mathcal P_{d,\beta}=\{(\vartheta a_1,\vartheta a_2,\vartheta(b_1+b_2)):(a_1,b_1),(a_2,b_2)\in\mathscr S_{d,\beta}\},$$
with $\vartheta$ as in \eqref{eq:CLM-product}. Finally, we set
\begin{equation}\label{eq:full-combined-region}
\mathcal C_{d,\beta}=\operatorname{conv}(\{(0,0,0)\}\cup\mathcal E_{d,\beta}\cup\mathcal J^{\mathrm{all}}_{d,\beta}\cup\mathcal L_{d,\beta}\cup\mathcal P_{d,\beta}).
\end{equation}
Each family has a simple origin. The points of $\mathcal E_{d,\beta}$ come from Theorem~\ref{thm:main_Sobolev_fractal_dilation} and Sobolev embedding, those of $\mathcal J^{\mathrm{all}}_{d,\beta}$ from transferring the fixed-radius bounds of Theorem~\ref{thm:asymmetric} to the maximal operator, as in Theorem~\ref{thm:mainasymmax}, those of $\mathcal L_{d,\beta}$ from the spherical maximal bounds applied to one input while the other one is bounded, and those of $\mathcal P_{d,\beta}$ from the product estimate \eqref{eq:CLM-product}. Since $\operatorname{conv}$ is not the closed convex hull, no strong-type endpoint is obtained merely by taking a limit of admissible exponents. Related refinements that use finer information about the geometry of the radius set appear in \cite{RoosSeeger2023}.

\begin{theorem}\label{thm:combined-localized}
Let $d\ge4$, and let $E\subset[1,2]$ be nonempty with $\overline{\dim}_{\mathrm M}E=\beta\in[0,1]$. Assume that $d>3+\beta$, and let $\mathcal C_{d,\beta}$ be the region defined in \eqref{eq:full-combined-region}. Then
$$\mathcal T_E:L^p(\mathbb R^d)\times L^q(\mathbb R^d)\longrightarrow L^r(\mathbb R^d)$$
whenever $(\frac1p,\frac1q,\frac1r)\in\mathcal C_{d,\beta}$.
\end{theorem}

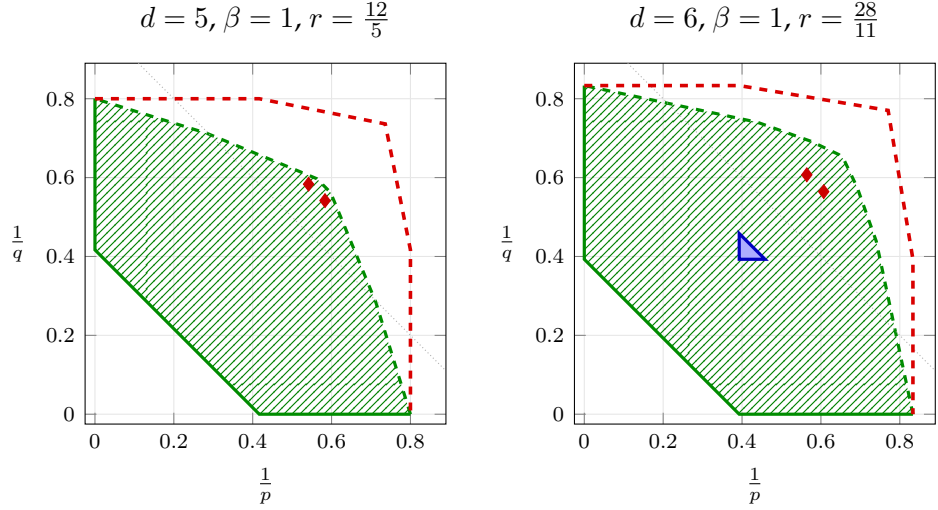
\begin{figure}[htbp]
\centering
\begin{tikzpicture}
\begin{groupplot}[
    group style={group size=2 by 1, horizontal sep=1.7cm},
    width=0.41\textwidth,
    height=0.41\textwidth,
    xmin=-0.025, xmax=0.89,
    ymin=-0.025, ymax=0.89,
    axis lines=box,
    xlabel={$\frac{1}{p}$},
    ylabel={$\frac{1}{q}$},
    ylabel style={rotate=-90},
    xtick={0,0.2,0.4,0.6,0.8},
    ytick={0,0.2,0.4,0.6,0.8},
    grid=major,
    grid style={gray!20},
    tick label style={font=\scriptsize},
    label style={font=\small},
    enlargelimits=false,
    clip=true
]

\nextgroupplot[title={$d=5$, $\beta=1$, $r=\frac{12}{5}$}]

\addplot[
    draw=none,
    fill=green!13,
    pattern=north east lines,
    pattern color=green!50!black,
    no marks
] coordinates {
    (0,{5/12})
    ({5/12},0)
    ({4/5},0)
    ({359/504},{143/504})
    ({209/336},{169/336})
    ({43/72},{9/16})
    ({9/16},{43/72})
    ({169/336},{209/336})
    ({143/504},{359/504})
    (0,{4/5})
} -- cycle;

\addplot[draw=green!55!black, very thick, no marks]
coordinates {
    (0,{4/5})
    (0,{5/12})
    ({5/12},0)
    ({4/5},0)
};

\addplot[draw=green!55!black, dashed, very thick, no marks]
coordinates {
    ({4/5},0)
    ({359/504},{143/504})
    ({209/336},{169/336})
    ({43/72},{9/16})
    ({9/16},{43/72})
    ({169/336},{209/336})
    ({143/504},{359/504})
    (0,{4/5})
};

\addplot[densely dotted, gray!65, domain=0.11:0.89, no marks]
{1-x};

\addplot[
    only marks,
    mark=diamond*,
    mark size=2.6pt,
    color=red!80!black
] coordinates {
    ({7/12},{13/24})
    ({13/24},{7/12})
};

\addplot[
    draw=red!85!black,
    dashed,
    line width=1.35pt,
    no marks,
    forget plot
] coordinates {
    (0,{4/5})
    ({5/12},{4/5})
    ({53/72},{53/72})
    ({4/5},{5/12})
    ({4/5},0)
};

\nextgroupplot[title={$d=6$, $\beta=1$, $r=\frac{28}{11}$}]

\addplot[
    draw=none,
    fill=green!13,
    pattern=north east lines,
    pattern color=green!50!black,
    no marks
] coordinates {
    (0,{11/28})
    ({11/28},0)
    ({5/6},0)
    ({151/204},{209/476})
    ({407/588},{563/980})
    ({129/196},{633/980})
    ({633/980},{129/196})
    ({563/980},{407/588})
    ({209/476},{151/204})
    (0,{5/6})
} -- cycle;

\addplot[draw=green!55!black, very thick, no marks]
coordinates {
    (0,{5/6})
    (0,{11/28})
    ({11/28},0)
    ({5/6},0)
};

\addplot[draw=green!55!black, dashed, very thick, no marks]
coordinates {
    ({5/6},0)
    ({151/204},{209/476})
    ({407/588},{563/980})
    ({129/196},{633/980})
    ({633/980},{129/196})
    ({563/980},{407/588})
    ({209/476},{151/204})
    (0,{5/6})
};

\addplot[
    draw=blue!75!black,
    fill=blue!32,
    very thick,
    no marks
] coordinates {
    ({11/28},{11/28})
    ({11/24},{11/28})
    ({11/28},{11/24})
} -- cycle;

\addplot[densely dotted, gray!65, domain=0.11:0.89, no marks]
{1-x};

\addplot[
    only marks,
    mark=diamond*,
    mark size=2.6pt,
    color=red!80!black
] coordinates {
    ({17/28},{79/140})
    ({79/140},{17/28})
};

\addplot[
    draw=red!85!black,
    dashed,
    line width=1.35pt,
    no marks,
    forget plot
] coordinates {
    (0,{5/6})
    ({11/28},{5/6})
    ({151/196},{151/196})
    ({5/6},{11/28})
    ({5/6},0)
};

\end{groupplot}
\end{tikzpicture}

\caption{Sections of the closure of $\mathcal C_{d,1}$ from
Theorem~\ref{thm:combined-localized}, which governs
$\mathcal T_{[1,2]}$, at a fixed output exponent in dimensions
$d=5$ and $d=6$, drawn in the $(\frac1p,\frac1q)$-plane.
Solid green boundary segments belong to the region except
at their junctions with the dashed green boundary, which
consists of limiting exponents at which no strong-type
estimate is claimed.
The red diamonds mark the asymmetric points of
Theorem~\ref{thm:mainasymmax} with $\theta=\frac12$, and
the red dashed lines show the upper and right boundaries
imposed by the necessary conditions of
Section~\ref{sec:necessary-conditions}.
The blue triangle in the right panel is the closure of
the region obtained by interpolating the Sobolev estimate
of \cite{BorgesFosterOu2025} with the
$L^\infty\times L^\infty\to L^\infty$ bound, and that
estimate does not apply when $d=5$.}
\label{fig:intro-max-compare-overview}
\end{figure}

Theorem~\ref{thm:combined-localized} also contains known one-sided estimates. The pointwise bound
$$\mathcal T_{[1,2]}(f,g)(x)\le\|g\|_\infty\sup_{1\le t\le2}A_t|f|(x)$$
and the classical spherical maximal theorem \cite{Stein1976} imply
$$\mathcal T_{[1,2]}:L^{\frac{12}{5}}(\mathbb R^5)\times L^\infty(\mathbb R^5)\longrightarrow L^{\frac{12}{5}}(\mathbb R^5),$$
and the corresponding reciprocal triple $(\frac{5}{12},0,\frac{5}{12})$ belongs to $\mathcal C_{5,1}$. It is the vertex $(\frac{5}{12},0)$ on the horizontal axis in the left panel of Figure~\ref{fig:intro-max-compare-overview}, and it shows how far $\mathcal C_{5,1}$ extends beyond the region of Theorem~\ref{thm:mainasymmax}.

\begin{remark}
The Sobolev smoothing and $L^p$-improving estimates obtained here may also be useful for sparse domination and for global maximal triangle operators. We plan to investigate these applications in subsequent work.
\end{remark}

\subsection{Main analytic difficulties}

The proofs have to overcome two analytic difficulties. The first is to obtain improved Sobolev smoothing estimates for triangle averages. Previous approaches, including those of Palsson and Sovine \cite{PalssonSovine2020} and of Borges, Foster and Ou \cite{BorgesFosterOu2025}, rely on decay estimates for the Fourier transform of the triangle surface measure, refined in the latter work by an angular decomposition, and this decay does not provide the smoothing gain needed here. We study the triangle incidence relation instead through its trilinear form and the theory of Fourier integral operators. The trilinear form can be read as a linear operator in three ways, by declaring one vertex of the triangle to be the output variable and the other two the input variables, and the obstruction is that none of the three readings is good everywhere: for each partition of the vertex variables, a cotangent projection of the conormal bundle fails to have maximal rank in certain conormal directions. We overcome this by adapting the microlocal partition method that Greenleaf, Iosevich and Taylor \cite{GreenleafIosevichTaylor2024} developed in their study of Mattila--Sj\"olin type interior results for configuration sets. At every nonzero conormal covector at least one of the three partitions $(1\mid23)$, $(2\mid13)$ and $(3\mid12)$ is nondegenerate in a conic neighborhood, for a reason that is easy to see in Figure~\ref{fig:multipliers}, so the conormal bundle is covered by three open cones $\mathcal W_j$, and on each of them the standard Sobolev estimates for Fourier integral operators apply after microlocalization. This yields a trilinear Sobolev estimate with a total gain of $\frac{d-3}{2}$ derivatives, uniformly over compact families of nondegenerate triangles, together with the weighted variants needed when the side lengths vary. Combined with a Littlewood--Paley decomposition and a covering argument for the set of radii, these estimates improve the maximal Sobolev smoothing bound by half a derivative.

The second difficulty is to obtain improved asymmetric $L^p\times L^q\to L^r$ estimates for the fixed-radius operator, and here our approach differs from the methods previously used for triangle averaging operators. We follow the analytic interpolation method behind the classical spherical $L^p$-improving estimates of Strichartz \cite{Strichartz1970} and Littman \cite{Littman1973}, but the analytic family is built by varying the shape of the triangle. We keep $|y|=|z|=1$ and replace the constraint $y\cdot z=\frac12$, which fixes the third side, by $y\cdot z=c$, so that the third side has length $\sqrt{2-2c}$. The corresponding averaging operator is denoted by $\mathcal U_c$, and $\mathcal U_{c_0}=\mathcal T_1$ for $c_0=\frac12$. We then deform analytically the distribution that imposes the constraint on the third side, setting
$$\mathcal A^\zeta(f,g)=\langle k_\zeta(s),\chi(s)\mathcal U_{c_0+s}(f,g)\rangle_s,$$
where $s$ is real, $\zeta$ is complex, $\chi$ is a smooth cutoff equal to one near the origin, and $k_\zeta(s)=\frac{|s|^{\zeta-1}}{\Gamma(\frac{\zeta}{2})}$ is understood by analytic continuation. Since $k_0=\delta_0$, the triangle average itself appears at $\zeta=0$.

The main issue in applying complex interpolation is the estimate on the left boundary of the strip; see Figure~\ref{fig:strip}. On the right boundary $\operatorname{Re}\zeta=1$ the family is controlled pointwise by a product of two spherical averages, and the sharp spherical $L^p$-improving estimate supplies the required bound. On the left boundary $\operatorname{Re}\zeta=-\frac{d-3}{2}$ we have to balance the Sobolev smoothing against derivatives in the shape parameter. Two parametrizations of the same family of triangles allow us to place these derivatives on whichever input has the lower frequency, while the smoothing gain is taken from the input with the higher frequency. Writing $A=\frac{d-3}{2}$ and $f_i$, $g_j$ for the Littlewood--Paley pieces of $f$ and $g$ at frequencies $2^i$ and $2^j$, the resulting estimate is
$$\|\partial_s^m[\chi(s)\mathcal U_{c_0+s}(f_i,g_j)]\|_2\lesssim_m2^{m\min(i,j)-A\max(i,j)}\|f_i\|_2\|g_j\|_2.$$
A scale-sensitive estimate that uses the homogeneity of the Riesz distributions then produces, on the critical left boundary, the dyadic factor
$$2^{A\min(i,j)-A\max(i,j)}=2^{-A|i-j|},$$
which is summable in the difference of the Littlewood--Paley indices when $d\ge4$. This summation is what preserves the strong-type endpoint, and complex interpolation, followed by the symmetry of the equilateral triangle, then yields the asymmetric estimates. Littlewood--Paley decompositions are also used in the maximal estimates, where they keep track of the smoothing gain, the derivatives in the radius and the loss coming from the covering number of the set of radii.

\subsection*{Organization of the paper}

In Section~\ref{sec:fixed} we establish uniform Sobolev estimates for triangle averages and use frequency localization and a covering argument to prove Theorem~\ref{thm:main_Sobolev_fractal_dilation}. In Section~\ref{sec:fixed-improving} we construct an analytic family of operators containing the triangle averages and prove Theorem~\ref{thm:asymmetric} and Corollary~\ref{cor:symmetric-fixed} by analytic interpolation, and we compare $\mathcal Q_d$ with the region $\widehat{\mathcal B}_d$. Section~\ref{sec:asymmax} develops a transfer principle from fixed-radius bounds to localized maximal estimates and proves Theorems~\ref{thm:mainasymmax} and~\ref{thm:combined-localized}. Finally, in Section~\ref{sec:necessary-conditions} we present necessary conditions for the fixed-radius and localized maximal estimates.

\section{Sobolev estimates for triangle averages}\label{sec:fixed}

We begin with a uniform Sobolev estimate for triangle averaging operators with a smoothing gain of $\frac{d-3}{2}$ derivatives. This estimate was obtained by Greenleaf, Iosevich and Taylor \cite[Section~5]{GreenleafIosevichTaylor2024} through a microlocal partition argument in their proof of Mattila--Sj\"olin type results for triangle configurations, and for completeness we give the details in the uniform form needed here. We then derive frequency-localized estimates and bounds for derivatives in the radii, and combining them with a covering argument we arrive at maximal Sobolev estimates over fractal sets of radii. These results are also used later, in the proofs of the fixed-radius and localized maximal $L^p$-improving estimates.

\subsection{Incidence geometry and microlocal analysis}

We work on the open set of triples $x=(x^1,x^2,x^3)\in(\mathbb R^d)^3$ with pairwise distinct entries, where the map
$$\Psi(x^1,x^2,x^3)=(|x^1-x^2|,|x^1-x^3|,|x^2-x^3|)$$
is smooth. Writing $e_{12}=\frac{x^1-x^2}{|x^1-x^2|}$, $e_{13}=\frac{x^1-x^3}{|x^1-x^3|}$ and $e_{23}=\frac{x^2-x^3}{|x^2-x^3|}$ for the unit vectors along the sides, the differential of $\Psi$ is the block matrix
$$D\Psi=\begin{pmatrix}
e_{12}&-e_{12}&0\\
e_{13}&0&-e_{13}\\
0&e_{23}&-e_{23}
\end{pmatrix},$$
whose rows are the gradients of the three side lengths and whose entries are row vectors in $\mathbb R^d$. Let
$$\Omega=\{(t_1,t_2,t_3)\in(0,\infty)^3:t_i<t_j+t_k\text{ whenever }\{i,j,k\}=\{1,2,3\}\}$$
be the set of side lengths of nondegenerate triangles. For $\vec t\in\Omega$ the strict triangle inequalities force every triple in $\Psi^{-1}(\vec t)$ to be noncollinear, so $e_{12}$ and $e_{13}$ are linearly independent and the three rows of $D\Psi$ are linearly independent as well. Consequently $Z_{\vec t}=\Psi^{-1}(\vec t)$ is a smooth submanifold of codimension three, and the distribution
$$\lambda_{\vec t}=\delta_0(\Psi-\vec t)$$
is a smooth density on $Z_{\vec t}$, that is, a smooth multiple of surface measure.

The triangle averages themselves live on the level sets of $F(y,z)=(|y|,|z|,|y-z|)$. For $\vec t\in\Omega$ let $\mathcal I_{\vec t}=\{(y,z)\in\mathbb R^{2d}:F(y,z)=\vec t\}$, so that $\mathcal I=\mathcal I_{(1,1,1)}$, and let $\sigma_{\vec t}$ be the normalized surface measure on $\mathcal I_{\vec t}$. By the coarea formula,
\begin{equation}\label{eq:probability-coarea}
d\sigma_{\vec t}(y,z)=m(\vec t)^{-1}\delta_0(F(y,z)-\vec t) dy dz,
\end{equation}
where $m(\vec t)=\int_{\mathbb R^{2d}}\delta_0(F(y,z)-\vec t) dy dz$. Indeed, $\mathcal I_{\vec t}$ is a single $O(d)$-orbit, and since $F$ is $O(d)$-invariant, the coarea Jacobian $\sqrt{\det(DF(DF)^{T})}$ is constant on this orbit and cancels when the measure is normalized. The function $m$ is positive and smooth on $\Omega$, and $m$ and $m^{-1}$ have bounded derivatives on compact subsets of $\Omega$.

For $a\in C_c^\infty((\mathbb R^d)^3)$ we define the localized trilinear form
$$\Lambda_{\vec t,a}(f_1,f_2,f_3)=\int_{(\mathbb R^d)^3}a(x)f_1(x^1)f_2(x^2)f_3(x^3) d\lambda_{\vec t}(x),$$
and for a smooth function $a$ on $(\mathbb R^d)^3$ we write
$$\|a\|_{C^N}=\sum_{|\alpha|\le N}\|\partial^\alpha a\|_{L^\infty((\mathbb R^d)^3)},$$
where $\alpha\in\mathbb N_0^{3d}$ ranges over multi-indices. The following estimate is obtained by adapting the microlocal partition argument of Greenleaf, Iosevich and Taylor \cite[Section~5]{GreenleafIosevichTaylor2024}.

\begin{theorem}\label{thm_trilinear_Sobolev_estimate}
Let $d\ge4$, and let $s_1,s_2,s_3\ge0$ satisfy
$$s_1+s_2+s_3\le\frac{d-3}{2}.$$
For every compact set $K\subset\Omega$ and every compact set $B\subset(\mathbb R^d)^3$ there exist an integer $N\ge0$ and a constant $C>0$ such that
\begin{equation}\label{eq:trilinear}
|\Lambda_{\vec t,a}(f_1,f_2,f_3)|\le C\|a\|_{C^N}\prod_{j=1}^3\|f_j\|_{H^{-s_j}}
\end{equation}
for all $\vec t\in K$, all $a\in C_c^\infty((\mathbb R^d)^3)$ supported in $B$ and all $f_1,f_2,f_3\in\mathcal S(\mathbb R^d)$.

Moreover, for every compact set $K\subset\Omega$ and every $R>0$ there exist an integer $N\ge0$ and a constant $C>0$, depending only on $d$, $K$, $R$ and $s_1,s_2,s_3$, such that
\begin{equation}\label{eq:trilinear-global}
|\Lambda_{\vec t,a}(f_1,f_2,f_3)|\le C\|a\|_{C^N}\prod_{j=1}^3\|f_j\|_{H^{-s_j}}
\end{equation}
for all $\vec t\in K$, all $f_1,f_2,f_3\in\mathcal S(\mathbb R^d)$ and all smooth functions $a$ on $(\mathbb R^d)^3$ with $\|a\|_{C^N}<\infty$ and
$$\operatorname{supp}a\subset\{(x^1,x^2,x^3):|x^2-x^1|\le R,|x^3-x^1|\le R\}.$$
\end{theorem}

\begin{proof}
We begin with \eqref{eq:trilinear}, so we fix $K$ and $B$. The vertices of every triple in $Z_{\vec t}$ with $\vec t\in K$ are at mutual distance at least $\min\{t_i:\vec t\in K,1\le i\le3\}>0$. Multiplying $a$ by a smooth cutoff that equals one near $\bigcup_{\vec t\in K}Z_{\vec t}$ and vanishes near the diagonals changes neither $\Lambda_{\vec t,a}$ nor the form of the estimate, so we may assume that $\Psi$ is smooth on a neighborhood of the support of $a$. There the delta distribution has the Fourier representation
$$\lambda_{\vec t}(x)=\int_{\mathbb R^3}e^{2\pi i\theta\cdot(\Psi(x)-\vec t)} d\theta$$
in the sense of distributions, and therefore
$$\Lambda_{\vec t,a}(f_1,f_2,f_3)=\int_{(\mathbb R^d)^3}K_{\vec t,a}(x)f_1(x^1)f_2(x^2)f_3(x^3) dx,$$
where the kernel $K_{\vec t,a}=a\lambda_{\vec t}$ is the oscillatory integral
$$K_{\vec t,a}(x)=\int_{\mathbb R^3}e^{2\pi i\theta\cdot(\Psi(x)-\vec t)}a(x) d\theta.$$
The phase function $\phi(x,\theta)=\theta\cdot(\Psi(x)-\vec t)$, with $x\in\mathbb R^{3d}$ and $\theta\in\mathbb R^3$, is nondegenerate because $D\Psi$ has rank three on $Z_{\vec t}$, and it parametrizes the conormal bundle
$$N^*Z_{\vec t}=\{(x,D\Psi(x)^{T}\theta):x\in Z_{\vec t},\theta\in\mathbb R^3\setminus\{0\}\}.$$
The amplitude has order zero and there are three phase variables in an ambient space of dimension $3d$, so in H\"ormander's notation \cite{Hormander1971,Hormander1985} we have $\lambda_{\vec t}\in I^{\frac32-\frac{3d}{4}}(\mathbb R^{3d};N^*Z_{\vec t})$.

To estimate $\Lambda_{\vec t,a}$ we regard $K_{\vec t,a}$ as the kernel of an operator that acts on two of the vertex variables and produces a function of the third. For $j\in\{1,2,3\}$ let $\{i_1,i_2\}=\{1,2,3\}\setminus\{j\}$ with $i_1<i_2$, and denote this partition of the vertex variables by $\sigma_j=(j\mid i_1i_2)$, so that the three partitions are $\sigma_1=(1\mid23)$, $\sigma_2=(2\mid13)$ and $\sigma_3=(3\mid12)$. The operator associated with $\sigma_j$ is
$$\widetilde{\mathcal T}^{\sigma_j}_{\vec t,a}h(x^j)=\iint K_{\vec t,a}(x^1,x^2,x^3)h(x^{i_1},x^{i_2}) dx^{i_1}dx^{i_2},$$
and for every $j$ we have
$$\Lambda_{\vec t,a}(f_1,f_2,f_3)=\int_{\mathbb R^d}\widetilde{\mathcal T}^{\sigma_j}_{\vec t,a}(f_{i_1}\otimes f_{i_2})(x^j)f_j(x^j) dx^j.$$
Write $\xi=(\xi^1,\xi^2,\xi^3)=D\Psi(x)^{T}\theta$ for the covectors in $N^*Z_{\vec t}$. Microlocally on the region where $\xi^j\ne0$, the operator $\widetilde{\mathcal T}^{\sigma_j}_{\vec t,a}$ is a Fourier integral operator of order $\frac32-\frac{3d}{4}$ from $\mathbb R^{2d}$ to $\mathbb R^d$, whose canonical relation is the twisted conormal bundle
$$C^{\sigma_j}_{\vec t}=\{(x^j,\xi^j;x^{i_1},x^{i_2},-\xi^{i_1},-\xi^{i_2}):(x,\xi)\in N^*Z_{\vec t},\xi^j\ne0\}.$$

We recall the Sobolev estimate that we intend to apply. A canonical relation $C\subset(T^*\mathbb R^d\setminus0)\times(T^*\mathbb R^{2d}\setminus0)$ is called nondegenerate if the natural projections $\pi_L:C\to T^*\mathbb R^d$ and $\pi_R:C\to T^*\mathbb R^{2d}$ have differentials of maximal rank. Since the output and input spaces have dimensions $d$ and $2d$, this means that $\pi_L$ is a submersion and $\pi_R$ is an immersion. For a Fourier integral operator $\mathcal R\in I^{m}(\mathbb R^d,\mathbb R^{2d};C)$ with compactly supported kernel we use, following \cite[(2.7)]{GreenleafIosevichTaylor2024}, the effective order
$$m_{\mathrm{eff}}=m+\frac{2d-d}{4}=m+\frac d4.$$
If $C$ is nondegenerate, the standard estimates for Fourier integral operators \cite{Hormander1971,Hormander1985} show that
\begin{equation}\label{eq:FIO_standard_estimate}
\mathcal R:H^s(\mathbb R^{2d})\longrightarrow H^{s-m_{\mathrm{eff}}}(\mathbb R^d)
\end{equation}
for every $s\in\mathbb R$. In our situation $m_{\mathrm{eff}}=\frac32-\frac{3d}{4}+\frac d4=-\frac{d-3}{2}$, which is the gain we are after. The difficulty is that none of the three canonical relations $C^{\sigma_j}_{\vec t}$ is nondegenerate on the whole conormal bundle, and our next task is to locate the bad directions.

For $\vec t\in\Omega$ let $c_{\vec t}=\frac{t_1^2+t_2^2-t_3^2}{2t_1t_2}$, the cosine of the angle of the triangle at the vertex $x^1$, and let
$$S_{\vec t}=\{(\omega,\widetilde\omega)\in S^{d-1}\times S^{d-1}:\omega\cdot\widetilde\omega=c_{\vec t}\}.$$
Since $\vec t\in\Omega$, we have $|c_{\vec t}|<1$, and
$$Z_{\vec t}=\{(x^1,x^1-t_1\omega,x^1-t_2\widetilde\omega):x^1\in\mathbb R^d,(\omega,\widetilde\omega)\in S_{\vec t}\}.$$
In these coordinates $x^1-x^2=t_1\omega$, $x^1-x^3=t_2\widetilde\omega$, and $x^2-x^3=t_2\widetilde\omega-t_1\omega$ is a vector of length $t_3$. Applying $D\Psi^{T}$ to $\theta=(\theta_1,\theta_2,\theta_3)\in\mathbb R^3\setminus\{0\}$, we see that the points of $N^*Z_{\vec t}$ are parametrized by $x^1\in\mathbb R^d$, $(\omega,\widetilde\omega)\in S_{\vec t}$ and $\theta\in\mathbb R^3\setminus\{0\}$, with
$$\begin{aligned}
(x^1,\xi^1)&=(x^1,\theta_1\omega+\theta_2\widetilde\omega),\\
(x^2,\xi^2)&=\left(x^1-t_1\omega,-\left(\theta_1+\frac{t_1}{t_3}\theta_3\right)\omega+\frac{t_2}{t_3}\theta_3\widetilde\omega\right),\\
(x^3,\xi^3)&=\left(x^1-t_2\widetilde\omega,\frac{t_1}{t_3}\theta_3\omega-\left(\theta_2+\frac{t_2}{t_3}\theta_3\right)\widetilde\omega\right).
\end{aligned}$$
The multiplier $\theta_k$ is attached to the $k$-th side of the triangle, and the covector at each vertex collects the contributions of the two sides through that vertex, as in Figure~\ref{fig:multipliers}.

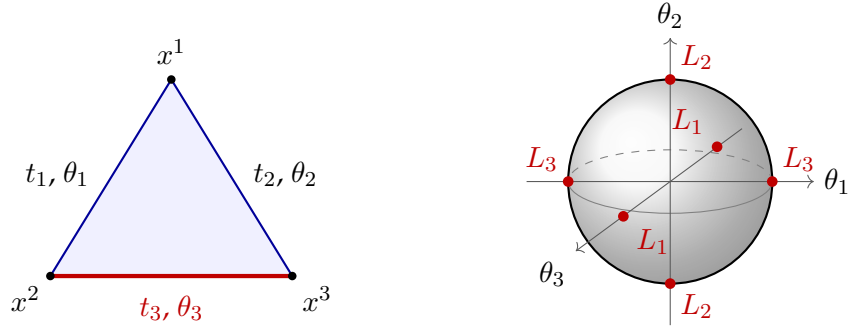
\begin{figure}[htbp]
\centering
\begin{tikzpicture}[scale=1]
\coordinate (A) at (0,2.6);
\coordinate (B) at (-1.6,0);
\coordinate (C) at (1.6,0);
\fill[blue!6] (A) -- (B) -- (C) -- cycle;
\draw[thick,blue!60!black] (A) -- (B);
\draw[thick,blue!60!black] (A) -- (C);
\draw[line width=1.6pt,red!75!black] (B) -- (C);
\fill (A) circle (1.6pt) node[above=2pt] {$x^1$};
\fill (B) circle (1.6pt) node[below left] {$x^2$};
\fill (C) circle (1.6pt) node[below right] {$x^3$};
\node[anchor=east] at ($(A)!0.5!(B)+(-0.15,0.05)$) {$t_1$, $\theta_1$};
\node[anchor=west] at ($(A)!0.5!(C)+(0.15,0.05)$) {$t_2$, $\theta_2$};
\node[anchor=north,red!75!black] at ($(B)!0.5!(C)+(0,-0.12)$) {$t_3$, $\theta_3$};
\begin{scope}[xshift=6.6cm,yshift=1.25cm]
\shade[ball color=gray!12,opacity=0.6] (0,0) circle (1.35);
\draw[thick] (0,0) circle (1.35);
\draw[gray] (-1.35,0) arc (180:360:1.35 and 0.42);
\draw[gray,dashed] (1.35,0) arc (0:180:1.35 and 0.42);
\draw[->,gray!70!black] (-1.9,0) -- (1.9,0) node[right,black] {$\theta_1$};
\draw[->,gray!70!black] (0,-1.9) -- (0,1.9) node[above,black] {$\theta_2$};
\draw[->,gray!70!black] (0.95,0.7) -- (-1.25,-0.92) node[below left,black] {$\theta_3$};
\fill[red!75!black] (1.35,0) circle (2pt) node[above right] {$L_3$};
\fill[red!75!black] (-1.35,0) circle (2pt) node[above left] {$L_3$};
\fill[red!75!black] (0,1.35) circle (2pt) node[above right] {$L_2$};
\fill[red!75!black] (0,-1.35) circle (2pt) node[below right] {$L_2$};
\fill[red!75!black] (-0.62,-0.46) circle (2pt);
\fill[red!75!black] (0.62,0.46) circle (2pt);
\node[red!75!black,anchor=north west] at (-0.58,-0.5) {$L_1$};
\node[red!75!black,anchor=south east] at (0.58,0.5) {$L_1$};
\end{scope}
\end{tikzpicture}
\caption{Left: the covector $D\Psi(x)^{T}\theta$ attaches the multiplier $\theta_k$ to the $k$-th side of the triangle, and its component $\xi^j$ at the vertex $x^j$ involves only the two sides through $x^j$. Hence $\xi^j=0$ if and only if $\theta$ is concentrated on the side opposite $x^j$. For $j=1$ this is the side $x^2x^3$ (red), which carries $\theta_3$, and the bad line is $L_1=\{\theta_1=\theta_2=0\}$. Right: on the unit sphere of $\theta$-space the partition $\sigma_j$ fails only at the two points of $L_j$, so every direction is good for at least two of the three partitions.}
\label{fig:multipliers}
\end{figure}

For $j=1,2,3$ let $\pi_j$ denote the projection of $N^*Z_{\vec t}$ onto the variables $(x^j,\xi^j)$, which is the left projection of $C^{\sigma_j}_{\vec t}$, and let us determine the rank of $D\pi_j$. The tangent vectors $(u,v)$ to $S_{\vec t}$ at $(\omega,\widetilde\omega)$ are characterized by the three conditions $u\cdot\omega=0$, $v\cdot\widetilde\omega=0$ and $u\cdot\widetilde\omega+\omega\cdot v=0$. Let $i_\omega:T_\omega S^{d-1}\hookrightarrow\mathbb R^d$ and $\widetilde i_{\widetilde\omega}:T_{\widetilde\omega}S^{d-1}\hookrightarrow\mathbb R^d$ be the natural inclusions. In the variables $(x^1,\omega,\widetilde\omega,\theta_1,\theta_2,\theta_3)$, the Jacobians, acting on the constrained tangent space just described, are
$$D\pi_1=\begin{pmatrix}
I_d&0&0&0&0&0\\
0&\theta_1i_\omega&\theta_2\widetilde i_{\widetilde\omega}&\omega&\widetilde\omega&0
\end{pmatrix},$$
$$D\pi_2=\begin{pmatrix}
I_d&-t_1i_\omega&0&0&0&0\\
0&-\left(\theta_1+\frac{t_1}{t_3}\theta_3\right)i_\omega&\frac{t_2}{t_3}\theta_3\widetilde i_{\widetilde\omega}&-\omega&0&\frac{t_2\widetilde\omega-t_1\omega}{t_3}
\end{pmatrix}$$
and
$$D\pi_3=\begin{pmatrix}
I_d&0&-t_2\widetilde i_{\widetilde\omega}&0&0&0\\
0&\frac{t_1}{t_3}\theta_3i_\omega&-\left(\theta_2+\frac{t_2}{t_3}\theta_3\right)\widetilde i_{\widetilde\omega}&0&-\widetilde\omega&\frac{t_1\omega-t_2\widetilde\omega}{t_3}
\end{pmatrix}.$$
Consider $D\pi_1$ first. Because the $x^1$-column of the second row vanishes, $D\pi_1$ is surjective if and only if its second row maps onto $\mathbb R^d$. The last three columns of that row contribute the plane spanned by $\omega$ and $\widetilde\omega$. If $\theta_1\ne0$, we may take $v=0$ and $u$ orthogonal to both $\omega$ and $\widetilde\omega$, which is a tangent vector to $S_{\vec t}$, and conclude that the second row also reaches every vector orthogonal to $\omega$ and $\widetilde\omega$; if $\theta_2\ne0$, the same argument works with the roles of $u$ and $v$ exchanged. Since $\omega$ and $\widetilde\omega$ are linearly independent, $D\pi_1$ is surjective unless $\theta_1=\theta_2=0$, and in that case the second row only reaches a plane, which is a proper subspace of $\mathbb R^d$. In $D\pi_2$ and $D\pi_3$ the constant columns again span the plane of $\omega$ and $\widetilde\omega$, and the same reasoning applies to the coefficients of $i_\omega$ and $\widetilde i_{\widetilde\omega}$. We conclude that $D\pi_j$ is surjective except on the line $L_j$, where
$$\begin{aligned}
L_1&=\{\theta_1=\theta_2=0\},\\
L_2&=\{\theta_1=\theta_3=0\},\\
L_3&=\{\theta_2=\theta_3=0\},
\end{aligned}$$
and that the output covector $\xi^j$ vanishes if and only if $\theta\in L_j$. Thus no single partition covers the entire conormal bundle.

On the other hand, the three lines meet only at the origin, so every $\theta\ne0$ lies outside at least two of them, and the open conic sets
$$\mathcal W_j=\{(x,D\Psi(x)^{T}\theta)\in N^*Z_{\vec t}:\theta\notin L_j\},$$
$j=1,2,3$, cover $N^*Z_{\vec t}$. On $\mathcal W_j$ the projection $\pi_j$ is a submersion, and the Lagrangian property of the canonical relation then forces the right projection to be an immersion there, because the kernel of the differential of the right projection is identified with the symplectic orthogonal complement of the range of the differential of the left projection. Hence $C^{\sigma_j}_{\vec t}$ is nondegenerate on $\mathcal W_j$. We also note that the input covector $(\xi^{i_1},\xi^{i_2})$ never vanishes when $\theta\ne0$, since this would place $\theta$ on two of the lines at once.

We now microlocalize in $\theta$. Choose $\gamma_1,\gamma_2,\gamma_3\in C^\infty(S^2)$ with $\gamma_1+\gamma_2+\gamma_3=1$ on the unit sphere $S^2\subset\mathbb R^3$ and with $\gamma_j$ vanishing in a neighborhood of the two points of $L_j\cap S^2$, which is possible because the sets $S^2\setminus L_j$ form an open cover of $S^2$. Choose also $\rho\in C_c^\infty(\mathbb R^3)$ equal to one near the origin. Splitting the $\theta$-integral accordingly, we write
$$K_{\vec t,a}=K^{\mathrm{sm}}_{\vec t,a}+\sum_{j=1}^3K^{j}_{\vec t,a},$$
where
$$K^{\mathrm{sm}}_{\vec t,a}(x)=a(x)\int_{\mathbb R^3}e^{2\pi i\theta\cdot(\Psi(x)-\vec t)}\rho(\theta) d\theta$$
and
$$K^{j}_{\vec t,a}(x)=\int_{\mathbb R^3}e^{2\pi i\theta\cdot(\Psi(x)-\vec t)}a(x)(1-\rho(\theta))\gamma_j\left(\frac{\theta}{|\theta|}\right) d\theta.$$
The integral in $K^{\mathrm{sm}}_{\vec t,a}$ is a Schwartz function evaluated at $\Psi(x)-\vec t$, so this kernel is smooth and compactly supported, and its derivatives are bounded by multiples of $\|a\|_{C^N}$ uniformly in $\vec t\in K$. The amplitude of $K^j_{\vec t,a}$ is a symbol of order zero whose conic support in $\theta$ stays a positive angular distance away from $L_j$, so $K^j_{\vec t,a}$ is the kernel of the Fourier integral operator
$$\mathcal R^{\sigma_j}_{\vec t,a}h(x^j)=\iint K^j_{\vec t,a}(x^1,x^2,x^3)h(x^{i_1},x^{i_2}) dx^{i_1}dx^{i_2},$$
whose canonical relation is the part of $C^{\sigma_j}_{\vec t}$ lying over a closed conic subset of $\mathcal W_j$. Writing $\langle F,G\rangle=\int FG$ for the bilinear pairing, we arrive at
\begin{equation}\label{eq:Lambda_decomposition}
\Lambda_{\vec t,a}(f_1,f_2,f_3)=\sum_{j=1}^3\langle\mathcal R^{\sigma_j}_{\vec t,a}(f_{i_1}\otimes f_{i_2}),f_j\rangle+\Lambda^{\mathrm{sm}}_{\vec t,a}(f_1,f_2,f_3),
\end{equation}
where $\Lambda^{\mathrm{sm}}_{\vec t,a}$ is the trilinear form with kernel $K^{\mathrm{sm}}_{\vec t,a}$.

By the computation above, each $\mathcal R^{\sigma_j}_{\vec t,a}$ has a nondegenerate canonical relation. Since $K$ is compact and the spatial support is compact, the estimate \eqref{eq:FIO_standard_estimate} holds uniformly in $\vec t\in K$, with constants controlled by finitely many seminorms of the amplitude \cite[Remark~2.2]{GreenleafIosevichTaylor2024}, and with $m_{\mathrm{eff}}=-\frac{d-3}{2}$ it takes the form
$$\|\mathcal R^{\sigma_j}_{\vec t,a}h\|_{H^{s+\frac{d-3}{2}}(\mathbb R^d)}\lesssim\|a\|_{C^N}\|h\|_{H^s(\mathbb R^{2d})}.$$
For nonnegative $s_{i_1}$ and $s_{i_2}$ the inequality $\langle(\xi,\eta)\rangle^{-2(s_{i_1}+s_{i_2})}\le\langle\xi\rangle^{-2s_{i_1}}\langle\eta\rangle^{-2s_{i_2}}$ and Plancherel's theorem show that
$$\|f_{i_1}\otimes f_{i_2}\|_{H^{-s_{i_1}-s_{i_2}}(\mathbb R^{2d})}\le\|f_{i_1}\|_{H^{-s_{i_1}}}\|f_{i_2}\|_{H^{-s_{i_2}}},$$
so the choice $h=f_{i_1}\otimes f_{i_2}$ and $s=-s_{i_1}-s_{i_2}$ yields
$$\|\mathcal R^{\sigma_j}_{\vec t,a}(f_{i_1}\otimes f_{i_2})\|_{H^{\frac{d-3}{2}-s_{i_1}-s_{i_2}}}\lesssim\|a\|_{C^N}\|f_{i_1}\|_{H^{-s_{i_1}}}\|f_{i_2}\|_{H^{-s_{i_2}}}.$$
The hypothesis $s_1+s_2+s_3\le\frac{d-3}{2}$ says that $\frac{d-3}{2}-s_{i_1}-s_{i_2}\ge s_j$, and pairing with $f_j\in H^{-s_j}$ we find
$$|\langle\mathcal R^{\sigma_j}_{\vec t,a}(f_{i_1}\otimes f_{i_2}),f_j\rangle|\lesssim\|a\|_{C^N}\prod_{i=1}^3\|f_i\|_{H^{-s_i}}.$$
The smooth term obeys the same bound, because its kernel belongs to $H^{s_1+s_2+s_3}(\mathbb R^{3d})$ with norm controlled by $\|a\|_{C^N}$, while the argument above shows that $\|f_1\otimes f_2\otimes f_3\|_{H^{-s_1-s_2-s_3}(\mathbb R^{3d})}\le\prod_i\|f_i\|_{H^{-s_i}}$. Summing the four contributions in \eqref{eq:Lambda_decomposition} proves \eqref{eq:trilinear}, uniformly for $\vec t\in K$.

We turn to \eqref{eq:trilinear-global}. Fix a compact set $K\subset\Omega$ and $R>0$, and let $a$ be as in the second part of the statement. Set $C_d=1+\sqrt d$ and choose $\chi\in C_c^\infty(B(0,C_d))$ with $\sum_{n\in\mathbb Z^d}\chi(x-n)=1$, which can be done by normalizing the translates of a nonnegative bump function that is positive on $[-\frac12,\frac12]^d$. Write $\chi_n(x)=\chi(x-n)$, choose $\eta\in C_c^\infty(B(0,C_d+1))$ equal to one on the support of $\chi$, and set $\eta_n(x)=\eta(x-n)$. Then
$$a(x^1,x^2,x^3)=\sum_{n,m,p\in\mathbb Z^d}a(x^1,x^2,x^3)\chi_n(x^1)\chi_m(x^2)\chi_p(x^3),$$
and because of the support condition on $a$ only the triples of indices of the form $(n,n+k,n+\ell)$ with $(k,\ell)$ in the finite set
$$F_R=\{(k,\ell)\in\mathbb Z^{2d}:|k|\le R+2C_d,|\ell|\le R+2C_d\}$$
contribute to the sum. Writing
$$a_{n,k,\ell}(x)=a(x)\chi_n(x^1)\chi_{n+k}(x^2)\chi_{n+\ell}(x^3),$$
we therefore have $a=\sum_{n\in\mathbb Z^d}\sum_{(k,\ell)\in F_R}a_{n,k,\ell}$, where each $a_{n,k,\ell}$ is supported in the product of closed balls $\overline{B(n,C_d)}\times\overline{B(n+k,C_d)}\times\overline{B(n+\ell,C_d)}$ and satisfies $\|a_{n,k,\ell}\|_{C^N}\lesssim_N\|a\|_{C^N}$ uniformly in $n$, $k$ and $\ell$. Since $\eta_n=1$ on the support of $\chi_n$,
$$\Lambda_{\vec t,a}(f_1,f_2,f_3)=\sum_{(k,\ell)\in F_R}\sum_{n\in\mathbb Z^d}\Lambda_{\vec t,a_{n,k,\ell}}(\eta_nf_1,\eta_{n+k}f_2,\eta_{n+\ell}f_3).$$
The incidence relation is invariant under simultaneous translation of the three vertices and Sobolev norms are translation invariant, so after a translation by $-n$ each term is controlled by \eqref{eq:trilinear}, applied with amplitudes supported in one of finitely many fixed compact sets. With $A_n=\|\eta_nf_1\|_{H^{-s_1}}$, $B_n=\|\eta_nf_2\|_{H^{-s_2}}$ and $C_n=\|\eta_nf_3\|_{H^{-s_3}}$ we arrive at
$$|\Lambda_{\vec t,a}(f_1,f_2,f_3)|\lesssim\|a\|_{C^N}\sum_{(k,\ell)\in F_R}\sum_{n\in\mathbb Z^d}A_nB_{n+k}C_{n+\ell}.$$
By the Cauchy--Schwarz inequality and the inclusion $\ell^2\subset\ell^\infty$,
$$\sum_{n\in\mathbb Z^d}A_nB_{n+k}C_{n+\ell}\le\|A\|_{\ell^2}\|B\|_{\ell^2}\|C\|_{\ell^2},$$
and the standard localization estimate $\sum_{n\in\mathbb Z^d}\|\eta_nf\|_{H^s}^2\lesssim_s\|f\|_{H^s}^2$, valid for every $s\in\mathbb R$, bounds the right-hand side by a multiple of $\prod_j\|f_j\|_{H^{-s_j}}$. Summing over the finite set $F_R$ proves \eqref{eq:trilinear-global} uniformly for $\vec t\in K$, which completes the proof of the theorem.
\end{proof}

\subsection{Frequency-localized estimates and radius derivatives}

We use a Littlewood--Paley decomposition. Choose $\varphi\in\mathcal S(\mathbb R^d)$ with $\widehat\varphi=1$ on $B(0,1)$ and $\operatorname{supp}\widehat\varphi\subset B(0,2)$, put $\widehat{\psi_0}=\widehat\varphi$ and
$$\widehat{\psi_i}(\xi)=\widehat\varphi(2^{-i}\xi)-\widehat\varphi(2^{-i+1}\xi)$$
for $i\ge1$, and write $P_if=f_i=\psi_i*f$ for the $i$-th Littlewood--Paley piece of $f$. The sum $\sum_{i=0}^n\widehat{\psi_i}(\xi)=\widehat\varphi(2^{-n}\xi)$ telescopes, so $f=\sum_{i\ge0}f_i$ in $\mathcal S$ for every $f\in\mathcal S$, and the Fourier supports satisfy $\operatorname{supp}\widehat{f_0}\subset B(0,2)$ and
$$\operatorname{supp}\widehat{f_i}\subset\{2^{i-1}\le|\xi|\le2^{i+1}\}$$
for $i\ge1$. By Plancherel's theorem and the bounded overlap of these supports,
\begin{equation}\label{eq:LP}
\|f\|_{H^{-s}}^2\approx\sum_{i\ge0}2^{-2si}\|f_i\|_2^2
\end{equation}
for every $s\in\mathbb R$. The support condition also implies the Bernstein inequality
\begin{equation}\label{eq:Bernstein}
\|\partial^\alpha f_i\|_2\lesssim_\alpha2^{i|\alpha|}\|f_i\|_2
\end{equation}
for all $i\ge0$ and $\alpha\in\mathbb N_0^d$, as well as the bound $\|F\|_{H^{-s}}\lesssim2^{-si}\|F\|_2$ for $s\ge0$ and every $F\in L^2$ whose Fourier transform is supported where $\widehat{f_i}$ is, because $\langle\xi\rangle\ge2^{i-1}$ there when $i\ge1$.

For a smooth function $a$ on $\mathcal I$ and $t_1,t_2>0$ we define the weighted triangle average
$$\mathcal T^a_{t_1,t_2}(f,g)(x)=\int_{\mathcal I}a(y,z)f(x-t_1y)g(x-t_2z) d\sigma(y,z),$$
and when $a\equiv1$ we simply write $\mathcal T_{t_1,t_2}$, so that $\mathcal T_{t,t}=\mathcal T_t$. Here and below $\|a\|_{C^N(\mathcal I)}$ is the $C^N$ norm of $a$ on the compact manifold $\mathcal I$, computed in a fixed finite atlas. The uniform trilinear Sobolev estimate of Theorem~\ref{thm_trilinear_Sobolev_estimate} implies the following frequency-localized bound.

\begin{lemma}\label{lem:weighted-fixed}
Let $d\ge4$, and let $K\subset(0,\infty)^2$ be compact. There exist an integer $N\ge0$ and a constant $C>0$, depending on $d$ and $K$, such that
\begin{equation}\label{eq:weighted-fixed}
\sup_{(t_1,t_2)\in K}\|\mathcal T^a_{t_1,t_2}(f_i,g_j)\|_2\le C\|a\|_{C^N(\mathcal I)}2^{-\frac{d-3}{2}\max(i,j)}\|f_i\|_2\|g_j\|_2
\end{equation}
for every $a\in C^\infty(\mathcal I)$, all $f,g\in L^2(\mathbb R^d)$ and all $i,j\ge0$.
\end{lemma}

\begin{proof}
Fix $(t_1,t_2)\in K$. Since $y\cdot z=\frac12$ on $\mathcal I$, the points $0$, $t_1y$ and $t_2z$ form a triangle with side lengths
$$\vec t(t_1,t_2)=\left(t_1,t_2,\sqrt{t_1^2+t_2^2-t_1t_2}\right),$$
and these side lengths lie in a compact subset of $\Omega$ as $(t_1,t_2)$ ranges over $K$. The map $(y,z)\mapsto(t_1y,t_2z)$ carries $\mathcal I$ onto $\mathcal I_{\vec t(t_1,t_2)}$ and commutes with rotations, so it pushes $\sigma$ forward to the invariant probability measure $\sigma_{\vec t(t_1,t_2)}$. Let $\widetilde a$ be a smooth compactly supported extension of $a$ to $\mathbb R^{2d}$ with $\|\widetilde a\|_{C^N(\mathbb R^{2d})}\lesssim_N\|a\|_{C^N(\mathcal I)}$ for every $N$, obtained for instance by composing $a$ with the nearest point projection onto $\mathcal I$ and multiplying by a cutoff, and set
$$b_{t_1,t_2}(x^1,x^2,x^3)=\widetilde a\left(\frac{x^1-x^2}{t_1},\frac{x^1-x^3}{t_2}\right).$$
The parameters $t_1$ and $t_2$ are bounded above and away from zero on $K$, so the functions $b_{t_1,t_2}$ satisfy the support condition of \eqref{eq:trilinear-global} for a fixed $R$, and their $C^N$ norms are bounded by a constant multiple of $\|a\|_{C^N(\mathcal I)}$. For $f,g,h\in\mathcal S(\mathbb R^d)$, the change of variables $(x^1,x^2,x^3)=(x,x-u,x-v)$ in the definition of $\Lambda$, together with \eqref{eq:probability-coarea} and the push-forward identity, shows that
$$\int_{\mathbb R^d}\mathcal T^a_{t_1,t_2}(f,g)(x)\overline{h(x)} dx=m(\vec t)^{-1}\Lambda_{\vec t,b_{t_1,t_2}}(\overline h,f,g),$$
where $\vec t=\vec t(t_1,t_2)$. Applying \eqref{eq:trilinear-global} with $(s_1,s_2,s_3)=(0,\frac{d-3}{2},0)$ and with $(s_1,s_2,s_3)=(0,0,\frac{d-3}{2})$, and using $L^2$ duality, we obtain
$$\|\mathcal T^a_{t_1,t_2}(f,g)\|_2\lesssim_K\|a\|_{C^N(\mathcal I)}\min\{\|f\|_{H^{-\frac{d-3}{2}}}\|g\|_2,\|f\|_2\|g\|_{H^{-\frac{d-3}{2}}}\},$$
uniformly for $(t_1,t_2)\in K$, and by density this holds for all $f,g\in L^2$. Replacing $f$ and $g$ by $f_i$ and $g_j$ and using $\|f_i\|_{H^{-s}}\lesssim2^{-si}\|f_i\|_2$ and $\|g_j\|_{H^{-s}}\lesssim2^{-sj}\|g_j\|_2$ with $s=\frac{d-3}{2}$ proves \eqref{eq:weighted-fixed}.
\end{proof}

\begin{remark}\label{rem:general-shapes}
The proof of Lemma~\ref{lem:weighted-fixed} uses only two features of the averages, namely that the side lengths of the triangles stay in a compact subset of $\Omega$ and that the weight is smooth. Consequently, for every compact set $K'\subset\Omega$ there is an integer $N$ such that the weighted averages
$$\int_{\mathcal I_{\vec t}}a(y,z)f(x-y)g(x-z) d\sigma_{\vec t}(y,z)$$
satisfy
$$\left\|\int_{\mathcal I_{\vec t}}a(y,z)f(x-y)g(x-z) d\sigma_{\vec t}(y,z)\right\|_{L^2(dx)}\lesssim_{K'}\|a\|_{C^N}\min\{\|f\|_{H^{-\frac{d-3}{2}}}\|g\|_2,\|f\|_2\|g\|_{H^{-\frac{d-3}{2}}}\}$$
uniformly for $\vec t\in K'$, where $a$ is smooth on a neighborhood of $\bigcup_{\vec t\in K'}\mathcal I_{\vec t}$ and the $C^N$ norm is taken over such a neighborhood. Moreover, since the bound in Lemma~\ref{lem:weighted-fixed} only uses the Fourier supports of $f_i$ and $g_j$, it remains valid when $f_i$ and $g_j$ are replaced by functions in $L^2$ whose Fourier transforms are supported in the same dyadic annuli, such as $\partial_\ell f_i$ and $\partial_mg_j$.
\end{remark}

We next estimate derivatives with respect to the radii. Differentiating the inputs with respect to their spatial variables introduces smooth weights on $\mathcal I$, which allows us to apply Lemma~\ref{lem:weighted-fixed} together with \eqref{eq:Bernstein}.

\begin{lemma}\label{lem:scale-derivative}
Let $d\ge4$, and let $K\subset(0,\infty)^2$ be compact. For all $f,g\in L^2(\mathbb R^d)$, all $i,j\ge0$ and all $\mu,\nu\in\{0,1\}$,
\begin{equation}\label{eq:derivative-biparam}
\sup_{(t_1,t_2)\in K}\|\partial_{t_1}^{\mu}\partial_{t_2}^{\nu}\mathcal T_{t_1,t_2}(f_i,g_j)\|_2\lesssim_K2^{\mu i+\nu j-\frac{d-3}{2}\max(i,j)}\|f_i\|_2\|g_j\|_2.
\end{equation}
In particular, for every compact interval $J\subset(0,\infty)$,
\begin{equation}\label{eq:derivative-single}
\sup_{t\in J}(\|\mathcal T_t(f_i,g_j)\|_2+2^{-\max(i,j)}\|\partial_t\mathcal T_t(f_i,g_j)\|_2)\lesssim_J2^{-\frac{d-3}{2}\max(i,j)}\|f_i\|_2\|g_j\|_2.
\end{equation}
\end{lemma}

\begin{proof}
Differentiating under the integral sign and applying the chain rule, we find
\begin{equation}\label{eq:radius-derivatives}
\begin{aligned}
\partial_{t_1}\mathcal T_{t_1,t_2}(f_i,g_j)&=-\sum_{\ell=1}^d\mathcal T^{y_\ell}_{t_1,t_2}(\partial_\ell f_i,g_j),\\
\partial_{t_2}\mathcal T_{t_1,t_2}(f_i,g_j)&=-\sum_{\ell=1}^d\mathcal T^{z_\ell}_{t_1,t_2}(f_i,\partial_\ell g_j),
\end{aligned}
\end{equation}
where $\mathcal T^{y_\ell}_{t_1,t_2}$ is the weighted average with weight $a(y,z)=y_\ell$, and similarly
$$\partial_{t_1}\partial_{t_2}\mathcal T_{t_1,t_2}(f_i,g_j)=\sum_{\ell,m=1}^d\mathcal T^{y_\ell z_m}_{t_1,t_2}(\partial_\ell f_i,\partial_mg_j).$$
The weights $y_\ell$, $z_m$ and $y_\ell z_m$ are smooth on $\mathcal I$, and the derivatives $\partial_\ell f_i$ and $\partial_mg_j$ have their Fourier transforms in the same annuli as $f_i$ and $g_j$. Lemma~\ref{lem:weighted-fixed}, in the form described in Remark~\ref{rem:general-shapes}, and the Bernstein inequality \eqref{eq:Bernstein} therefore yield \eqref{eq:derivative-biparam}.

Finally, since $\mathcal T_t=\mathcal T_{t,t}$, the chain rule yields
$$\partial_t\mathcal T_t(f_i,g_j)=(\partial_{t_1}+\partial_{t_2})\mathcal T_{t_1,t_2}(f_i,g_j),$$
where the right-hand side is evaluated at $t_1=t_2=t$, and since $2^i+2^j\le2^{\max(i,j)+1}$, the estimate \eqref{eq:derivative-biparam} implies \eqref{eq:derivative-single}.
\end{proof}

\subsection{Maximal estimates over sets of radii}

We obtain maximal estimates by combining the frequency-localized bounds of the preceding subsection with a covering argument adapted from \cite{BorgesFosterOu2025}. The idea is that on an interval of length $\delta$ a function of $t$ cannot move much more than $\delta$ times its derivative, so its supremum over the interval is controlled by its average and the average of its derivative.

\begin{lemma}\label{lem:covering}
Let $1\le r<\infty$, let $E\subset[1,2]$ be nonempty, and let $0<\delta\le\frac14$. Suppose that $F(t,x)$ is complex-valued and jointly measurable on $[\frac12,\frac52]\times\mathbb R^d$, that it is continuously differentiable in $t\in[\frac12,\frac52]$ for almost every $x$, and that for some constant $B\ge0$
\begin{equation}\label{eq:bound_F_derivative_F}
\sup_{\frac12\le t\le\frac52}(\|F(t,\cdot)\|_r+\delta\|\partial_tF(t,\cdot)\|_r)\le B.
\end{equation}
Then
$$\left\|\sup_{t\in E}|F(t,\cdot)|\right\|_r\le3N(E,\delta)^{\frac1r}B.$$
\end{lemma}

\begin{proof}
The supremum over $E$ may be taken over a countable dense subset of $E$, because $F$ is continuous in $t$ for almost every $x$, so it is measurable. Let $\{I_\nu\}_{\nu=1}^N$ be a cover of $E$ by $N=N(E,\delta)$ intervals of length $\delta$, each meeting $E$, and let $2I_\nu$ be the interval with the same center and twice the length. Since $E\subset[1,2]$ and $\delta\le\frac14$, we have $2I_\nu\subset[\frac12,\frac52]$.

Fix $\nu$. By the fundamental theorem of calculus, for $t\in I_\nu$ and $s\in2I_\nu$,
$$|F(t,x)|\le|F(s,x)|+\int_{2I_\nu}|\partial_uF(u,x)| du,$$
and averaging over $s\in2I_\nu$ leads to
$$\sup_{t\in I_\nu}|F(t,x)|\le\frac{1}{|2I_\nu|}\int_{2I_\nu}|F(s,x)| ds+\int_{2I_\nu}|\partial_sF(s,x)| ds.$$
Minkowski's inequality and \eqref{eq:bound_F_derivative_F} then show that
$$\left\|\sup_{t\in I_\nu}|F(t,\cdot)|\right\|_r\le\sup_{s\in2I_\nu}\|F(s,\cdot)\|_r+2\delta\sup_{s\in2I_\nu}\|\partial_sF(s,\cdot)\|_r\le3B.$$
Consequently,
$$\left\|\sup_{t\in E}|F(t,\cdot)|\right\|_r^r\le\sum_{\nu=1}^{N(E,\delta)}\left\|\sup_{t\in I_\nu}|F(t,\cdot)|\right\|_r^r\le N(E,\delta)(3B)^r,$$
and taking the $r$-th root proves the lemma.
\end{proof}

For $f,g\in\mathcal S(\mathbb R^d)$ and $i,j\ge0$ we define the frequency-localized maximal operator
$$\mathcal T_E^{i,j}(f,g)(x)=\sup_{t\in E}|\mathcal T_t(f_i,g_j)(x)|.$$
For Schwartz functions the series $\sum_{i,j}\mathcal T_t(f_i,g_j)(x)$ converges absolutely to $\mathcal T_t(f,g)(x)$, so the triangle inequality implies
\begin{equation}\label{eq:TE_decomposition}
\mathcal T_E(f,g)(x)\le\sum_{i,j\ge0}\mathcal T_E^{i,j}(f,g)(x).
\end{equation}
Lemmas~\ref{lem:covering} and~\ref{lem:scale-derivative} together provide $L^2$ estimates for the frequency-localized maximal operators in terms of the covering numbers of $E$.

\begin{proposition}\label{prop:pieces}
Let $d\ge4$, and let $E\subset[1,2]$ be nonempty with $\overline{\dim}_{\mathrm M}E=\beta\in[0,1]$. For all $f,g\in\mathcal S(\mathbb R^d)$ and all $i,j\ge0$,
\begin{equation}\label{eq:pieces-cover}
\|\mathcal T_E^{i,j}(f,g)\|_2\lesssim N(E,2^{-\max(i,j)})^{\frac12}2^{-\frac{d-3}{2}\max(i,j)}\|f_i\|_2\|g_j\|_2.
\end{equation}
In particular, for every $\varepsilon>0$,
\begin{equation}\label{eq:pieces-decay}
\|\mathcal T_E^{i,j}(f,g)\|_2\lesssim_{E,\varepsilon}2^{-\frac{d-3-\beta-\varepsilon}{2}\max(i,j)}\|f_i\|_2\|g_j\|_2.
\end{equation}
\end{proposition}

\begin{proof}
Let $k=\max(i,j)$ and $\delta=2^{-k-2}$, and apply Lemma~\ref{lem:covering} with $r=2$ to $F(t,x)=\mathcal T_t(f_i,g_j)(x)$. Since $\delta\le2^{-k}$, the estimate \eqref{eq:derivative-single} with $J=[\frac12,\frac52]$ shows that we may take
$$B=C2^{-\frac{d-3}{2}k}\|f_i\|_2\|g_j\|_2.$$
Every interval of length $2^{-k}$ is covered by four intervals of length $2^{-k-2}$, so $N(E,2^{-k-2})\le4N(E,2^{-k})$, and \eqref{eq:pieces-cover} follows. Inserting \eqref{eq:minkowski-cover} into \eqref{eq:pieces-cover} gives \eqref{eq:pieces-decay}.
\end{proof}

We are now ready to prove Theorem~\ref{thm:main_Sobolev_fractal_dilation}.

\begin{proof}[Proof of Theorem~\ref{thm:main_Sobolev_fractal_dilation}]
Choose $\varepsilon>0$ so small that
$$\rho=\frac{d-3-\beta-\varepsilon}{2}-s_1-s_2>0,$$
and let $f,g\in\mathcal S(\mathbb R^d)$. By \eqref{eq:TE_decomposition} and \eqref{eq:pieces-decay},
$$\|\mathcal T_E(f,g)\|_2\lesssim_{E,\varepsilon}\sum_{i,j\ge0}2^{-\frac{d-3-\beta-\varepsilon}{2}\max(i,j)}\|f_i\|_2\|g_j\|_2.$$
Put $a_i=2^{-s_1i}\|f_i\|_2$ and $b_j=2^{-s_2j}\|g_j\|_2$. Since $s_1,s_2\ge0$ and $s_1i+s_2j\le(s_1+s_2)\max(i,j)$, each term of the last sum is at most $2^{-\rho\max(i,j)}a_ib_j$, and the Cauchy--Schwarz inequality leads to
$$\|\mathcal T_E(f,g)\|_2\lesssim_{E,\varepsilon}\left(\sum_{i,j\ge0}2^{-2\rho\max(i,j)}\right)^{\frac12}\left(\sum_{i\ge0}a_i^2\right)^{\frac12}\left(\sum_{j\ge0}b_j^2\right)^{\frac12}.$$
The first factor is finite, because $2\max(i,j)\ge i+j$ and therefore
$$\sum_{i,j\ge0}2^{-2\rho\max(i,j)}\le\left(\sum_{i\ge0}2^{-\rho i}\right)^2<\infty,$$
and by \eqref{eq:LP} the other two factors are comparable to $\|f\|_{H^{-s_1}}$ and $\|g\|_{H^{-s_2}}$. Hence
$$\|\mathcal T_E(f,g)\|_2\lesssim\|f\|_{H^{-s_1}}\|g\|_{H^{-s_2}}$$
for Schwartz functions $f$ and $g$. For Schwartz functions $f,g,f',g'$ the pointwise inequality $|\mathcal T_E(f,g)-\mathcal T_E(f',g')|\le\mathcal T_E(f-f',g)+\mathcal T_E(f',g-g')$, combined with this estimate, yields
$$\|\mathcal T_E(f,g)-\mathcal T_E(f',g')\|_2\lesssim\|f-f'\|_{H^{-s_1}}\|g\|_{H^{-s_2}}+\|f'\|_{H^{-s_1}}\|g-g'\|_{H^{-s_2}},$$
so $\mathcal T_E$ extends uniquely by continuity to $H^{-s_1}(\mathbb R^d)\times H^{-s_2}(\mathbb R^d)$ with the same bound.
\end{proof}

The Sobolev embedding argument used in the proof of \cite[Proposition~6.2]{BorgesFosterOu2025} leads to the following $L^p$-improving consequence.

\begin{corollary}\label{thm:lp-improving}
Let $E\subset[1,2]$ be nonempty with upper Minkowski dimension $\beta$, and suppose that $d>3+\beta$. If $1<p_1,p_2\le2$ satisfy
\begin{equation}\label{eq:lp-improving-range}
\frac{1}{p_1}+\frac{1}{p_2}<1+\frac{d-3-\beta}{2d},
\end{equation}
then
$$\|\mathcal T_E(f,g)\|_2\lesssim\|f\|_{p_1}\|g\|_{p_2}.$$
\end{corollary}

\begin{proof}
Set $s_i=d(\frac{1}{p_i}-\frac12)$ for $i=1,2$. Since $1<p_i\le2$, the dual form of the Sobolev embedding provides the inclusion $L^{p_i}(\mathbb R^d)\hookrightarrow H^{-s_i}(\mathbb R^d)$; see \cite[Chapter~1]{GrafakosModern2014}. Moreover, \eqref{eq:lp-improving-range} implies
$$s_1+s_2=d\left(\frac{1}{p_1}+\frac{1}{p_2}-1\right)<\frac{d-3-\beta}{2},$$
and the conclusion follows from Theorem~\ref{thm:main_Sobolev_fractal_dilation}.
\end{proof}

\clearpage
\section{\texorpdfstring{Fixed-radius $L^p$-improving estimates}{Fixed-radius Lp-improving estimates}}\label{sec:fixed-improving}

In this section we prove asymmetric $L^p$-improving estimates for the fixed-radius triangle averaging operator $\mathcal T_1$. Following the classical analytic interpolation method for spherical averages \cite{Strichartz1970}, we introduce an analytic family of operators, establish bounds for this family on the two edges of a strip, and apply complex interpolation to obtain Theorem~\ref{thm:asymmetric}.

\subsection{An analytic family of triangle averages}

We vary the shape of the triangle while keeping two of its side lengths equal to one, and the resulting family of averaging operators contains $\mathcal T_1$ as its equilateral member. For $-1<c<1$ let
$$\mathcal I(c)=\{(y,z)\in\mathbb R^{2d}:|y|=|z|=1,y\cdot z=c\},$$
let $\sigma_c$ be the normalized surface measure on $\mathcal I(c)$, and define
$$\mathcal U_c(f,g)(x)=\int_{\mathcal I(c)}f(x-y)g(x-z) d\sigma_c(y,z).$$
Since $|y-z|=\sqrt{2-2c}$ on $\mathcal I(c)$, these operators average over nondegenerate triangles with side lengths $(1,1,\sqrt{2-2c})$, and in particular $\mathcal U_{\frac12}=\mathcal T_1$. As for $\mathcal I$, the group $O(d)$ acts transitively on $\mathcal I(c)$, and $\sigma_c$ is its unique invariant probability measure.

To differentiate with respect to the shape parameter $c$, we parametrize $\mathcal I(c)$ by a fixed manifold. Let
$$V_2(\mathbb R^d)=\{(y,w):|y|=|w|=1,y\cdot w=0\}$$
be the manifold of orthonormal pairs, with its rotation-invariant probability measure $\nu_2$, and for $-1<c<1$ put
$$z_c=cy+\sqrt{1-c^2}w.$$
Since $y$ and $w$ are orthogonal unit vectors, $|z_c|^2=c^2+(1-c^2)=1$ and $y\cdot z_c=c$, so $(y,z_c)\in\mathcal I(c)$. Conversely, for every $(y,z)\in\mathcal I(c)$ the vector $w=\frac{z-cy}{\sqrt{1-c^2}}$ satisfies $|w|=1$ and $y\cdot w=0$. Hence $(y,w)\mapsto(y,z_c)$ is a bijection of $V_2(\mathbb R^d)$ onto $\mathcal I(c)$, and since it commutes with rotations, it pushes $\nu_2$ forward to the invariant probability measure $\sigma_c$. Consequently,
\begin{equation}\label{eq:shapeparam_2}
\mathcal U_c(f,g)(x)=\int_{V_2(\mathbb R^d)}f(x-y)g(x-z_c) d\nu_2(y,w).
\end{equation}
Now the domain of integration and the measure no longer depend on $c$, so differentiation in $c$ produces spatial derivatives of $g$ with smooth coefficients. Alternatively, we may fix $z$ and write $y=cz+\sqrt{1-c^2}w$, which places these derivatives on $f$ instead. Both parametrizations represent the same operator, and Figure~\ref{fig:shape-family} illustrates the first one. In particular, $c\mapsto\mathcal U_c(f,g)(x)$ is smooth on $(-1,1)$ for Schwartz functions $f$ and $g$.

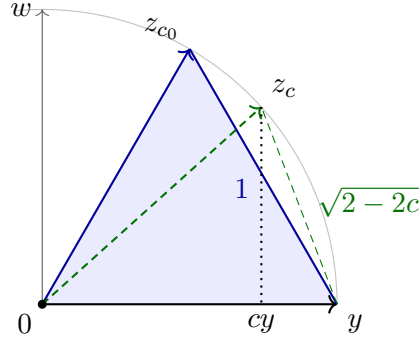
\begin{figure}[htbp]
\centering
\begin{tikzpicture}[scale=1.3]
\coordinate (O) at (0,0);
\coordinate (Y) at (3,0);
\coordinate (W) at (0,3);
\coordinate (Z) at ({3*cos(60)},{3*sin(60)});
\coordinate (Zc) at ({3*cos(42)},{3*sin(42)});
\draw[gray!50] (3,0) arc (0:95:3);
\fill[blue!8] (O) -- (Y) -- (Z) -- cycle;
\draw[->,thick] (O) -- (Y);
\draw[->,gray!70!black] (O) -- (W);
\draw[->,thick,blue!60!black] (O) -- (Z);
\draw[thick,blue!60!black] (Y) -- (Z);
\draw[->,thick,green!45!black,densely dashed] (O) -- (Zc);
\draw[green!45!black,densely dashed] (Y) -- (Zc);
\draw[dotted,thick] (Zc) -- ({3*cos(42)},0);
\fill (O) circle (1.3pt) node[below left] {$0$};
\node[below right] at (Y) {$y$};
\node[left] at (W) {$w$};
\node[above left] at (Z) {$z_{c_0}$};
\node[above right] at (Zc) {$z_c$};
\node[below] at ({3*cos(42)},0) {$cy$};
\node[blue!60!black] at ($(Y)!0.5!(Z)+(-0.22,-0.12)$) {$1$};
\node[green!45!black,anchor=west] at ($(Y)!0.5!(Zc)+(0.08,0.04)$) {$\sqrt{2-2c}$};
\end{tikzpicture}
\caption{The family $\mathcal I(c)$ consists of pairs of unit vectors with $y\cdot z=c$, so that the triangle with vertices $0$, $y$ and $z$ has third side of length $\sqrt{2-2c}$. The equilateral member $c_0=\frac12$, whose third side has length one, is shaded, and a second member of the family is dashed. Writing $z_c=cy+\sqrt{1-c^2}w$ with $w$ a unit vector orthogonal to $y$ fixes the domain of integration, so derivatives in $c$ fall on $g$, while fixing $z$ instead puts them on $f$.}
\label{fig:shape-family}
\end{figure}

Fix $c_0=\frac12$ and $0<\epsilon_0<\frac18$, so that $c_0+s$ stays in $(\frac38,\frac58)$ when $|s|<\epsilon_0$, and choose an even function $\chi\in C_c^\infty((-\epsilon_0,\epsilon_0))$ with $\chi\equiv1$ near zero. To build an analytic family through $\mathcal U_{c_0}=\mathcal T_1$, we use the normalized one-dimensional Riesz distributions
$$k_\zeta(s)=\frac{|s|^{\zeta-1}}{\Gamma(\frac{\zeta}{2})},$$
initially defined for $\operatorname{Re}\zeta>0$ as locally integrable functions acting on compactly supported smooth test functions. We record the standard properties of these distributions that we need; see \cite[Chapter~I, Sections~3.3 and~3.5]{GelfandShilov1964} and \cite[Section~3.2]{HormanderI2003}.

\begin{lemma}\label{lem:Riesz-continuation}
The family $k_\zeta$, $\operatorname{Re}\zeta>0$, admits an entire distribution-valued continuation with $k_0=\delta_0$, which has the following two properties. First, for every $\zeta\in\mathbb C$, every $\delta>0$ and every $H\in C_c^\infty(\mathbb R)$,
\begin{equation}\label{eq:Riesz-homogeneity}
\langle k_\zeta(s),H(s)\rangle=\delta^\zeta\langle k_\zeta(t),H(\delta t)\rangle.
\end{equation}
Second, let $J\subset\mathbb R$ be a compact interval, let $R>0$, and let $M\ge0$ be an integer with $M+u>0$ for all $u\in J$. Then there exist a constant $C>0$ and an integer $N\ge0$ such that
\begin{equation}\label{eq:Riesz-growth}
|\langle k_{u+i\tau},H\rangle|\le C(1+|\tau|)^Ne^{\frac{\pi}{4}|\tau|}\sum_{m=0}^M\|H^{(m)}\|_\infty
\end{equation}
for all $u\in J$, all $\tau\in\mathbb R$ and all $H\in C_c^\infty((-R,R))$.
\end{lemma}

\begin{proof}
For $\operatorname{Re}\zeta>0$ and $H\in C_c^\infty((-R,R))$ we have
$$\langle k_\zeta,H\rangle=\frac{1}{\Gamma(\frac{\zeta}{2})}\int_0^Rs^{\zeta-1}G(s) ds,$$
where $G(s)=H(s)+H(-s)$ is even. Fix an even function $\eta\in C_c^\infty((-R,R))$ with $\eta=1$ on $[-\frac R2,\frac R2]$ and an integer $M\ge0$, and subtract from $G$ its Taylor polynomial of order $M-1$ at the origin, which contains only even powers of $s$. The remainder
$$r_M(s)=G(s)-2\eta(s)\sum_{0\le2n<M}\frac{H^{(2n)}(0)}{(2n)!}s^{2n}$$
satisfies $|r_M(s)|\le C_{R,M}s^M\sum_{m=0}^M\|H^{(m)}\|_\infty$ for $0<s<R$, by Taylor's theorem on $[0,\frac R2]$ and trivially on $[\frac R2,R)$. Consequently,
\begin{equation}\label{eq:Taylor_subtraction_formula}
\langle k_\zeta,H\rangle=\frac{1}{\Gamma(\frac{\zeta}{2})}\int_0^Rs^{\zeta-1}r_M(s) ds+2\sum_{0\le2n<M}\frac{H^{(2n)}(0)}{(2n)!}\cdot\frac{J_n(\zeta)}{\Gamma(\frac{\zeta}{2})},
\end{equation}
where $J_n(\zeta)=\int_0^Rs^{\zeta+2n-1}\eta(s) ds$. The integral involving $r_M$ converges absolutely and is holomorphic for $\operatorname{Re}\zeta>-M$. Since $\eta=1$ on $[0,\frac R2]$, we have
$$J_n(\zeta)=\frac{1}{\zeta+2n}\left(\frac R2\right)^{\zeta+2n}+\int_{\frac R2}^Rs^{\zeta+2n-1}\eta(s) ds,$$
so $J_n$ is meromorphic with a single simple pole at $\zeta=-2n$, and this pole is cancelled by the zero of the entire function $\frac{1}{\Gamma(\frac{\zeta}{2})}$ at the same point. The right-hand side of \eqref{eq:Taylor_subtraction_formula} is therefore holomorphic in the half-plane $\operatorname{Re}\zeta>-M$, and it coincides with $\langle k_\zeta,H\rangle$ when $\operatorname{Re}\zeta>0$. The continuations obtained for different values of $M$ agree on their common domains, and since $M$ is arbitrary, $\langle k_\zeta,H\rangle$ extends to an entire function of $\zeta$. The estimates below show that for each $\zeta$ this extension depends continuously on finitely many derivatives of $H$, so it defines a distribution.

To identify $k_0$, take $M=1$. At $\zeta=0$ the integral involving $r_1$ is regular and is multiplied by $\frac{1}{\Gamma(0)}=0$, while $J_0(\zeta)=\frac1\zeta+O(1)$ and $\frac{1}{\Gamma(\frac{\zeta}{2})}=\frac\zeta2+O(\zeta^2)$. Hence $\langle k_0,H\rangle=2H(0)\cdot\frac12=H(0)$, that is, $k_0=\delta_0$. The homogeneity relation \eqref{eq:Riesz-homogeneity} follows from the substitution $s=\delta t$ when $\operatorname{Re}\zeta>0$, and it extends to all $\zeta$ because both sides are entire.

For the growth bound, write $\zeta=u+i\tau$ with $u\in J$ and use \eqref{eq:Taylor_subtraction_formula} with the given $M$. Since $u+M$ stays positive and bounded away from zero on the compact interval $J$, the integral involving $r_M$ is bounded by a constant multiple of $\sum_{m\le M}\|H^{(m)}\|_\infty$. Stirling's formula shows that
$$\left|\frac{1}{\Gamma(\frac{u+i\tau}{2})}\right|\lesssim_J(1+|\tau|)^Ne^{\frac\pi4|\tau|},$$
and the entire functions $\frac{J_n(\zeta)}{\Gamma(\frac{\zeta}{2})}$ obey the same bound, because $J_n$ is bounded on $\{u+i\tau:u\in J,|u+i\tau+2n|\ge1\}$ while the complementary set is compact. Since $|H^{(2n)}(0)|\le\|H^{(2n)}\|_\infty$, this proves \eqref{eq:Riesz-growth}.
\end{proof}

For Schwartz functions $f$ and $g$ and each $x\in\mathbb R^d$, the function $s\mapsto\chi(s)\mathcal U_{c_0+s}(f,g)(x)$ is smooth and supported in $(-\epsilon_0,\epsilon_0)$, and here and below we extend it by zero outside this interval. We may therefore define the analytic family
$$\mathcal A^\zeta(f,g)(x)=\langle k_\zeta(s),\chi(s)\mathcal U_{c_0+s}(f,g)(x)\rangle_s.$$
By Lemma~\ref{lem:Riesz-continuation}, $\zeta\mapsto\mathcal A^\zeta(f,g)(x)$ is entire for every $x$, and since $k_0=\delta_0$ and $\chi(0)=1$,
$$\mathcal A^0(f,g)=\mathcal U_{c_0}(f,g)=\mathcal T_1(f,g).$$
When $\operatorname{Re}\zeta>0$, the operator $\mathcal A^\zeta$ is an average of the operators $\mathcal U_c$ against the weight $\chi(c-c_0)k_\zeta(c-c_0)$ in the shape parameter. Decreasing $\operatorname{Re}\zeta$ makes this weight more singular at $c=c_0$, and at $\zeta=0$ it collapses to the point mass at the equilateral triangle.

\subsection{The left boundary estimate}

We now establish the bound
$$\mathcal A^{-\frac{d-3}{2}+i\tau}:L^2(\mathbb R^d)\times L^2(\mathbb R^d)\longrightarrow L^2(\mathbb R^d)$$
for all $\tau\in\mathbb R$, with constants of at most exponential growth in $|\tau|$. On this line the Riesz distribution acts as a derivative of order $\frac{d-3}{2}$ in the shape parameter, and the main step is to pay for these derivatives with the Sobolev smoothing of Section~\ref{sec:fixed}. We place the shape derivatives on the input with lower frequency, while the smoothing gain comes from the input with higher frequency, and together with a scale-sensitive estimate for the Riesz distributions this produces enough decay to sum the dyadic frequency pieces. We first control derivatives of the triangle averages with respect to the shape parameter.

\begin{lemma}\label{lem:shapederivatives}
Let $d\ge4$, let $f,g\in\mathcal S(\mathbb R^d)$ and $i,j\ge0$, and put $k=\max(i,j)$ and $\ell=\min(i,j)$. For every integer $m\ge0$,
\begin{equation}\label{eq:shapederivatives}
\sup_{s\in\mathbb R}\|\partial_s^m\{\chi(s)\mathcal U_{c_0+s}(f_i,g_j)\}\|_2\lesssim_m2^{m\ell-\frac{d-3}{2}k}\|f_i\|_2\|g_j\|_2,
\end{equation}
where the implicit constant depends on $d$, $m$ and the fixed cutoff $\chi$.
\end{lemma}

\begin{proof}
As $c$ ranges over $c_0+\operatorname{supp}\chi\subset(\frac38,\frac58)$, the side lengths $(1,1,\sqrt{2-2c})$ stay in a compact subset of $\Omega$, so Remark~\ref{rem:general-shapes} applies uniformly to the averages over $\mathcal I(c)$ with smooth weights. We use the parametrization \eqref{eq:shapeparam_2}. Differentiating in $c$, we find
$$\partial_cz_c=y-\frac{c}{\sqrt{1-c^2}}w=\frac{y-cz_c}{1-c^2},$$
so with $v(c,y,z)=\frac{y-cz}{1-c^2}$ and differentiation under the integral sign, which is legitimate because the domain and the measure in \eqref{eq:shapeparam_2} do not depend on $c$, we obtain
$$\begin{aligned}
\partial_c\mathcal U_c(f_i,g_j)(x)&=-\int_{V_2(\mathbb R^d)}f_i(x-y)\nabla g_j(x-z_c)\cdot\partial_cz_c d\nu_2(y,w)\\
&=-\int_{\mathcal I(c)}f_i(x-y)v(c,y,z)\cdot\nabla g_j(x-z) d\sigma_c(y,z).
\end{aligned}$$
Thus a derivative in $c$ produces a first-order derivative of $g_j$ multiplied by a smooth coefficient. To handle higher derivatives, let $a(c,y,z)$ be a smooth coefficient. By the chain rule,
$$\frac{d}{dc}a(c,y,z_c)=((\partial_c+v(c,y,z)\cdot\nabla_z)a)(c,y,z_c),$$
and therefore the derivative in $c$ of $\int_{\mathcal I(c)}a(c,y,z)f_i(x-y)\partial^\alpha g_j(x-z) d\sigma_c(y,z)$ equals
$$\begin{aligned}
&\int_{\mathcal I(c)}f_i(x-y)((\partial_c+v\cdot\nabla_z)a)(c,y,z)\partial^\alpha g_j(x-z) d\sigma_c(y,z)\\
&{}-\sum_{r=1}^d\int_{\mathcal I(c)}f_i(x-y)a(c,y,z)v_r(c,y,z)\partial^{\alpha+e_r}g_j(x-z) d\sigma_c(y,z),
\end{aligned}$$
where $e_r$ is the multi-index with a one in the $r$-th coordinate and zeros elsewhere. Each differentiation either falls on a coefficient or adds one derivative to $g_j$, so by induction $\partial_c^m\mathcal U_c(f_i,g_j)$ is a finite sum of terms of the form
\begin{equation}\label{eq:shape-derivative-terms}
\int_{\mathcal I(c)}a_{m,\alpha}(c,y,z)f_i(x-y)\partial^\alpha g_j(x-z) d\sigma_c(y,z)
\end{equation}
with $|\alpha|\le m$. The coefficients arise by differentiating $v(c,y,z)=\frac{y-cz}{1-c^2}$, so they are polynomials in $(y,z)$ whose coefficients are smooth functions of $c$ on $(-1,1)$. Since $c$ stays in a compact subinterval of $(-1,1)$, where $1-c^2$ is bounded away from zero, the coefficients $a_{m,\alpha}$ and their derivatives of any fixed order are uniformly bounded on a neighborhood of the relevant manifolds. Derivatives falling on $\chi(s)$, with $c=c_0+s$, produce terms of the same form.

We estimate each term \eqref{eq:shape-derivative-terms} by Remark~\ref{rem:general-shapes}, assigning the full Sobolev exponent $\frac{d-3}{2}$ to one of the two inputs. The resulting bound is
$$\lesssim_m\min\{\|f_i\|_{H^{-\frac{d-3}{2}}}\|\partial^\alpha g_j\|_2,\|f_i\|_2\|\partial^\alpha g_j\|_{H^{-\frac{d-3}{2}}}\}\lesssim2^{-\frac{d-3}{2}k}\|f_i\|_2\|\partial^\alpha g_j\|_2,$$
because $f_i$ and $\partial^\alpha g_j$ have their Fourier transforms at frequencies of size $2^i$ and $2^j$. By the Bernstein inequality \eqref{eq:Bernstein}, $\|\partial^\alpha g_j\|_2\lesssim_m2^{j|\alpha|}\|g_j\|_2\le2^{mj}\|g_j\|_2$ for $|\alpha|\le m$, and after summing the finitely many terms we are left with
$$\sup_{s\in\mathbb R}\|\partial_s^m\{\chi(s)\mathcal U_{c_0+s}(f_i,g_j)\}\|_2\lesssim_m2^{mj-\frac{d-3}{2}k}\|f_i\|_2\|g_j\|_2.$$
The second parametrization, which keeps $z$ fixed, moves all spatial derivatives to $f_i$, and the same argument yields this bound with $2^{mi}$ in place of $2^{mj}$. Both parametrizations represent the same operator, so we may use the smaller of the two bounds, and this is \eqref{eq:shapederivatives}.
\end{proof}

We next estimate the action of a Riesz distribution on a Banach-valued function. By Lemma~\ref{lem:shapederivatives}, the functions we have in mind satisfy bounds of the form $\sup_s\|\partial_s^mH(s)\|_X\le BL^m$ for $0\le m\le M$, with $L=2^{\min(i,j)}$. To pair such a function with $k_{-a+i\tau}$ through Lemma~\ref{lem:Riesz-continuation} we need an integer $M>a$, and a direct application of \eqref{eq:Riesz-growth} costs $M$ derivatives, hence a factor $BL^M$, which is too large for the endpoint frequency summation. The homogeneity of the Riesz distribution allows us to replace $L^M$ by $L^a$.

\begin{lemma}\label{lem:rieszscale}
Let $X$ be a Banach space, let $a>0$, and let $M>a$ be an integer. Suppose that $H\in C_c^\infty(\mathbb R;X)$ is supported in $[-1,1]$ and that for some $B>0$ and $L\ge1$
$$\sup_{s\in\mathbb R}\|\partial_s^mH(s)\|_X\le BL^m$$
for $0\le m\le M$. Then there exist a constant $C>0$ and an integer $N\ge0$, depending only on $a$ and $M$, such that
\begin{equation}\label{eq:rieszscale}
\|\langle k_{-a+i\tau},H\rangle\|_X\le C(1+|\tau|)^Ne^{\frac\pi4|\tau|}BL^a
\end{equation}
for all $\tau\in\mathbb R$.
\end{lemma}

\begin{proof}
For $X$-valued functions the pairing with $k_\zeta$ is defined by the same formulas as in the scalar case, and for every $x^*$ in the dual space $X^*$ we have $x^*(\langle k_\zeta,H\rangle)=\langle k_\zeta,x^*\circ H\rangle$. Since the norm of an element of $X$ is the supremum of $|x^*(\cdot)|$ over the unit ball of $X^*$, the scalar estimates of Lemma~\ref{lem:Riesz-continuation} apply to $X$-valued functions with the same constants, and we use them in this form.

Set $\delta=L^{-1}$, choose $\phi\in C_c^\infty((-2,2))$ with $0\le\phi\le1$ and $\phi=1$ on $[-1,1]$, and split $H=H_0+H_1$, where $H_0(s)=\phi(\frac s\delta)H(s)$. By the homogeneity identity \eqref{eq:Riesz-homogeneity},
$$\langle k_{-a+i\tau},H_0\rangle=\delta^{-a+i\tau}\langle k_{-a+i\tau}(t),F(t)\rangle,$$
where $F(t)=\phi(t)H(\delta t)$, and $|\delta^{i\tau}|=1$. By the product and chain rules, for $0\le m\le M$,
$$\partial_t^mF(t)=\sum_{r=0}^m\binom mr\phi^{(m-r)}(t)\delta^rH^{(r)}(\delta t),$$
and the hypothesis $\sup_s\|H^{(r)}(s)\|_X\le BL^r$ together with $\delta L=1$ yields
$$\sup_t\|\partial_t^mF(t)\|_X\le B\sum_{r=0}^m\binom mr\|\phi^{(m-r)}\|_\infty(\delta L)^r\lesssim_mB.$$
Thus the derivatives of $F$ up to order $M$ are bounded independently of $L$, and $F$ is supported in the fixed interval $(-2,2)$. Applying \eqref{eq:Riesz-growth} with $J=\{-a\}$ and $R=2$, we obtain
$$\|\langle k_{-a+i\tau},H_0\rangle\|_X\le\delta^{-a}\|\langle k_{-a+i\tau},F\rangle\|_X\lesssim(1+|\tau|)^Ne^{\frac\pi4|\tau|}BL^a.$$

The function $H_1(s)=(1-\phi(\frac s\delta))H(s)$ vanishes for $|s|\le\delta$, so its pairing with $k_\zeta$ is the absolutely convergent integral $\frac{1}{\Gamma(\frac\zeta2)}\int|s|^{\zeta-1}H_1(s) ds$ for every $\zeta$, both sides being entire and equal when $\operatorname{Re}\zeta>0$. Using $\|H_1(s)\|_X\le\|H(s)\|_X\le B$, we find
$$\|\langle k_{-a+i\tau},H_1\rangle\|_X\le\left|\frac{1}{\Gamma(\frac{-a+i\tau}{2})}\right|\int_{|s|\ge\delta}|s|^{-a-1}B ds=\frac{2B}{a}\left|\frac{1}{\Gamma(\frac{-a+i\tau}{2})}\right|\delta^{-a},$$
and by Stirling's formula the reciprocal Gamma factor grows at most like $(1+|\tau|)^Ne^{\frac\pi4|\tau|}$. The function $\frac1\Gamma$ is entire, so this bound remains valid where the reciprocal Gamma factor vanishes. Adding the bounds for $H_0$ and $H_1$ and recalling that $\delta^{-a}=L^a$ proves \eqref{eq:rieszscale}.
\end{proof}

Lemmas~\ref{lem:shapederivatives} and~\ref{lem:rieszscale} combine to give the desired estimate on the left boundary.

\begin{proposition}\label{prop:edgeL2}
Let $d\ge4$. For every $0<a\le\frac{d-3}{2}$ there exist a constant $C_{a,d}>0$ and an integer $N\ge0$, depending only on $a$, $d$ and the fixed cutoff $\chi$, such that for all $f,g\in L^2(\mathbb R^d)$
$$\|\mathcal A^{-a+i\tau}(f,g)\|_2\le C_{a,d}(1+|\tau|)^Ne^{\frac\pi4|\tau|}\|f\|_2\|g\|_2$$
for every $\tau\in\mathbb R$. In particular, the endpoint $a=\frac{d-3}{2}$ is included.
\end{proposition}

\begin{proof}
Let $f,g\in\mathcal S(\mathbb R^d)$. We first treat the frequency-localized inputs $f_i$ and $g_j$. Choose an integer $M>\frac{d-3}{2}$, put $H(s)=\chi(s)\mathcal U_{c_0+s}(f_i,g_j)$ and $L=2^{\min(i,j)}$, and let
$$B=C_M2^{-\frac{d-3}{2}\max(i,j)}\|f_i\|_2\|g_j\|_2.$$
We regard $H$ as a function with values in $X=L^2(\mathbb R^d)$, which is legitimate because $f_i$ and $g_j$ are smooth with rapidly decaying derivatives, so that $H\in C_c^\infty(\mathbb R;L^2(\mathbb R^d))$. With a suitable constant $C_M$, Lemma~\ref{lem:shapederivatives} ensures that $\sup_s\|\partial_s^mH(s)\|_2\le BL^m$ for $0\le m\le M$, so the hypotheses of Lemma~\ref{lem:rieszscale} are satisfied. Since $\mathcal A^{-a+i\tau}(f_i,g_j)=\langle k_{-a+i\tau},H\rangle$, that lemma yields
\begin{equation}\label{eq:edgefrequency}
\|\mathcal A^{-a+i\tau}(f_i,g_j)\|_2\le C_{a,d}(\tau)2^{a\min(i,j)-\frac{d-3}{2}\max(i,j)}\|f_i\|_2\|g_j\|_2,
\end{equation}
where $C_{a,d}(\tau)\le C_{a,d}(1+|\tau|)^Ne^{\frac\pi4|\tau|}$.

We next sum over the frequency indices. Put $u_i=\|f_i\|_2$ and $v_j=\|g_j\|_2$ for $i,j\ge0$, and extend both sequences by zero to negative indices. For Schwartz inputs the series $\sum_{i,j}\mathcal A^{-a+i\tau}(f_i,g_j)$ converges absolutely to $\mathcal A^{-a+i\tau}(f,g)$, because \eqref{eq:Riesz-growth} bounds $|\mathcal A^{-a+i\tau}(f_i,g_j)(x)|$ by finitely many derivatives of $f_i$ and $g_j$, which decay rapidly in $i$ and $j$. Bilinearity, the triangle inequality and \eqref{eq:edgefrequency} therefore lead to
$$\|\mathcal A^{-a+i\tau}(f,g)\|_2\le\sum_{i,j}\|\mathcal A^{-a+i\tau}(f_i,g_j)\|_2\le C_{a,d}(\tau)\sum_{i,j}2^{a\min(i,j)-\frac{d-3}{2}\max(i,j)}u_iv_j.$$
Write $k=\max(i,j)$ and $\ell=\min(i,j)$. Since $k=\ell+|i-j|$ and $a\le\frac{d-3}{2}$,
$$2^{a\ell-\frac{d-3}{2}k}=2^{-\frac{d-3}{2}|i-j|}2^{-(\frac{d-3}{2}-a)\ell}\le2^{-\frac{d-3}{2}|i-j|},$$
and hence
\begin{equation}\label{eq:lefbound_ij}
\|\mathcal A^{-a+i\tau}(f,g)\|_2\le C_{a,d}(\tau)\sum_{i,j}2^{-\frac{d-3}{2}|i-j|}u_iv_j.
\end{equation}
This is where the hypothesis $d\ge4$ enters, since it makes the sequence $w_n=2^{-\frac{d-3}{2}|n|}$, $n\in\mathbb Z$, summable, with
$$\|w\|_{\ell^1(\mathbb Z)}=1+2\sum_{n=1}^\infty2^{-\frac{d-3}{2}n}=\frac{1+2^{-\frac{d-3}{2}}}{1-2^{-\frac{d-3}{2}}}<\infty.$$
By the Cauchy--Schwarz inequality and the discrete Young inequality,
$$\sum_{i,j\ge0}2^{-\frac{d-3}{2}|i-j|}u_iv_j=\sum_{i\ge0}u_i(w*v)_i\le\|u\|_{\ell^2}\|w*v\|_{\ell^2}\le\|w\|_{\ell^1}\|u\|_{\ell^2}\|v\|_{\ell^2}\lesssim\|f\|_2\|g\|_2,$$
where the last step uses Plancherel's theorem and the bounded overlap of the Littlewood--Paley Fourier supports. Together with \eqref{eq:lefbound_ij} this proves the estimate for Schwartz inputs, and it extends to $L^2$ inputs by density.
\end{proof}

\subsection{The right boundary estimate}

On the line $\operatorname{Re}\zeta=1$ the Riesz kernel $k_\zeta$ is a bounded function, and the family $\mathcal A^\zeta$ becomes an honest average over triangles of all shapes near the equilateral one. Such an average is dominated by the product of two spherical averages, which is the content of the next lemma.

\begin{lemma}\label{lem:right-boundary}
Let $d\ge4$. There is a constant $C_d>0$, depending only on $d$ and the fixed cutoff $\chi$, such that
$$\|\mathcal A^{1+i\tau}(f,g)\|_{\frac{d+1}{2}}\le C_de^{C_d|\tau|}\|f\|_{\frac{d+1}{d}}\|g\|_{\frac{d+1}{d}}$$
for all $f,g\in\mathcal S(\mathbb R^d)$ and all $\tau\in\mathbb R$.
\end{lemma}

\begin{proof}
Since $k_\zeta$ is locally integrable when $\operatorname{Re}\zeta>0$, the change of variables $c=c_0+s$ turns the definition of $\mathcal A^\zeta$ into
$$\mathcal A^\zeta(f,g)(x)=\int_{-1}^1\chi(c-c_0)k_\zeta(c-c_0)\mathcal U_c(f,g)(x) dc,$$
so $\mathcal A^\zeta$ is a weighted integral of triangle averages over the shape parameter. To compare it with spherical averages, we apply the coarea formula to the map $F(y,z)=y\cdot z$ from $S^{d-1}\times S^{d-1}$ to $[-1,1]$. For $c\in(-1,1)$ the level set $F^{-1}(c)$ is $\mathcal I(c)$, the coarea Jacobian depends only on $c$, and by rotational invariance the normalized measure on each level set is $\sigma_c$. Consequently,
\begin{equation}\label{eq:angular-disintegration}
\int_{S^{d-1}\times S^{d-1}}G(y,z) d\sigma_{S^{d-1}}(y)d\sigma_{S^{d-1}}(z)=\int_{-1}^1\omega_d(c)\int_{\mathcal I(c)}G(y,z) d\sigma_c(y,z)dc
\end{equation}
for every integrable function $G$, where
$$\omega_d(c)=\frac{\Gamma(\frac d2)}{\sqrt\pi\Gamma(\frac{d-1}{2})}(1-c^2)^{\frac{d-3}{2}}$$
is the density of the inner product of two independent uniformly distributed unit vectors, a smooth positive function on $(-1,1)$. For $\operatorname{Re}\zeta>0$ we apply \eqref{eq:angular-disintegration} to $G(y,z)=f(x-y)g(x-z)K_\zeta(y,z)$, where
$$K_\zeta(y,z)=\frac{\chi(y\cdot z-c_0)k_\zeta(y\cdot z-c_0)}{\omega_d(y\cdot z)},$$
which is possible because $\omega_d^{-1}$ is bounded on the compact set $c_0+\operatorname{supp}\chi\subset(-1,1)$. The result is
$$\mathcal A^\zeta(f,g)(x)=\int_{S^{d-1}\times S^{d-1}}f(x-y)g(x-z)K_\zeta(y,z) d\sigma_{S^{d-1}}(y)d\sigma_{S^{d-1}}(z).$$
On the line $\zeta=1+i\tau$ we have $|k_{1+i\tau}(s)|=\frac{1}{|\Gamma(\frac{1+i\tau}{2})|}$ for $s\ne0$, so $|K_{1+i\tau}(y,z)|\le C(\tau)$, where $C(\tau)$ is a constant multiple of $\frac{1}{|\Gamma(\frac{1+i\tau}{2})|}$, and
$$\begin{aligned}
|\mathcal A^{1+i\tau}(f,g)(x)|&\le C(\tau)\int_{S^{d-1}\times S^{d-1}}|f(x-y)||g(x-z)| d\sigma_{S^{d-1}}(y)d\sigma_{S^{d-1}}(z)\\
&=C(\tau)A_1|f|(x)A_1|g|(x).
\end{aligned}$$
The classical spherical $L^p$-improving estimate of Strichartz \cite{Strichartz1970} states that $\|A_1h\|_{d+1}\lesssim\|h\|_{\frac{d+1}{d}}$, so H\"older's inequality yields
$$\|\mathcal A^{1+i\tau}(f,g)\|_{\frac{d+1}{2}}\le C(\tau)\|A_1|f|\|_{d+1}\|A_1|g|\|_{d+1}\lesssim C(\tau)\|f\|_{\frac{d+1}{d}}\|g\|_{\frac{d+1}{d}}.$$
By Stirling's formula, $C(\tau)$ grows at most exponentially in $|\tau|$.
\end{proof}

\subsection{Interpolation of positive averages and the comparison region}\label{sec:regions}

To combine estimates we use the following elementary interpolation lemma for positive averages. It remains valid when the output exponent is less than one, which makes it suitable for the quasi-Banach estimates in $\mathcal B_d$ and in $\mathcal P_{d,\beta}$.

\begin{lemma}\label{lem:positive-interpolation}
Let $\mathcal B$ be a positive bilinear average, that is, an operator of the form
$$\mathcal B(f,g)(x)=\int_{\mathbb R^{2d}}f(x-y)g(x-z) d\mu(y,z)$$
for a finite positive Borel measure $\mu$, or the supremum of a countable family of such averages. Suppose that $\mathcal B$ satisfies strong-type bounds at the reciprocal triples $(\frac{1}{p_j},\frac{1}{q_j},\frac{1}{r_j})$, $j=0,1$, where $1\le p_j,q_j\le\infty$ and $0<r_j\le\infty$. Then $\mathcal B$ satisfies a strong-type bound at every point of the segment joining these triples.
\end{lemma}

\begin{proof}
Let $0<\lambda<1$, and define $p$, $q$ and $r$ by $\frac1p=\frac{1-\lambda}{p_0}+\frac{\lambda}{p_1}$, $\frac1q=\frac{1-\lambda}{q_0}+\frac{\lambda}{q_1}$ and $\frac1r=\frac{1-\lambda}{r_0}+\frac{\lambda}{r_1}$. Consider first bounded nonnegative functions $f$ and $g$ supported on sets of finite measure, and assume that $p,q<\infty$. Put $F_j=f^{\frac{p}{p_j}}$ and $G_j=g^{\frac{q}{q_j}}$, where a zeroth power is interpreted as the indicator function of the support. Then $f=F_0^{1-\lambda}F_1^\lambda$ and $g=G_0^{1-\lambda}G_1^\lambda$, and H\"older's inequality with exponents $\frac{1}{1-\lambda}$ and $\frac1\lambda$ with respect to $\mu$ shows that
$$\mathcal B(f,g)\le\mathcal B(F_0,G_0)^{1-\lambda}\mathcal B(F_1,G_1)^\lambda$$
pointwise, an inequality that survives taking suprema over a family of averages. Next we apply H\"older's inequality in the spatial variable with exponents $\frac{r_0}{r(1-\lambda)}$ and $\frac{r_1}{r\lambda}$. These exponents are at least one even when $r<1$, because $\frac{r(1-\lambda)}{r_0}+\frac{r\lambda}{r_1}=1$, and since $\|F_j\|_{p_j}=\|f\|_p^{\frac{p}{p_j}}$ and $\|G_j\|_{q_j}=\|g\|_q^{\frac{q}{q_j}}$, we obtain
$$\|\mathcal B(f,g)\|_r\le\|\mathcal B(F_0,G_0)\|_{r_0}^{1-\lambda}\|\mathcal B(F_1,G_1)\|_{r_1}^\lambda\lesssim\|f\|_p^{\frac{(1-\lambda)p}{p_0}+\frac{\lambda p}{p_1}}\|g\|_q^{\frac{(1-\lambda)q}{q_0}+\frac{\lambda q}{q_1}}=\|f\|_p\|g\|_q.$$
If $p=\infty$, then $p_0=p_1=\infty$ and we take $F_0=F_1=f$, and similarly when $q=\infty$, while infinite values of $r_0$, $r_1$ or $r$ are handled with the usual conventions. The general case follows from the inequality $|\mathcal B(f,g)|\le\mathcal B(|f|,|g|)$ and monotone convergence.
\end{proof}

The second tool is the symmetry of the equilateral triangle, which we now verify.

\begin{lemma}\label{lem:symmetry}
The trilinear form
$$\Lambda(f_1,f_2,f_3)=\int_{\mathbb R^d}\mathcal T_1(f_1,f_2)(x)f_3(x) dx$$
is invariant under all permutations of $f_1$, $f_2$ and $f_3$ for $f_1,f_2,f_3\in\mathcal S(\mathbb R^d)$.
\end{lemma}

\begin{proof}
The transpositions of the first two arguments and of the first and third arguments generate the symmetric group on three letters, so it suffices to check these two. The map $(y,z)\mapsto(z,y)$ is a bijection of $\mathcal I$ that commutes with rotations, so it preserves the invariant probability measure $\sigma$, and it exchanges $f_1$ and $f_2$. For the second transposition, we substitute $x'=x-y$ in
$$\Lambda(f_1,f_2,f_3)=\int_{\mathbb R^d}\int_{\mathcal I}f_1(x-y)f_2(x-z)f_3(x) d\sigma(y,z)dx$$
and set $(Y,Z)=(-y,z-y)$, so that $x-z=x'-Z$ and $x=x'-Y$. The map $(y,z)\mapsto(Y,Z)$ carries $\mathcal I$ to itself, because $|Y|=|y|=1$, $|Z|=|z-y|=1$ and $|Y-Z|=|z|=1$. It is a bijection, with inverse $(Y,Z)\mapsto(-Y,Z-Y)$, and it commutes with rotations, so it preserves $\sigma$ as well. Therefore
$$\Lambda(f_1,f_2,f_3)=\int_{\mathbb R^d}\int_{\mathcal I}f_1(x')f_2(x'-Z)f_3(x'-Y) d\sigma(Y,Z)dx'=\Lambda(f_3,f_2,f_1).$$
\end{proof}

With these two lemmas we can justify the claim made in the introduction that $\mathcal T_1$ is bounded at every point of $\widehat{\mathcal B}_d$. The seven points generating $\mathcal B_d$ correspond to the bounds listed there, namely the elementary estimates, the spherical estimate and the estimates of \cite{IosevichPalssonSovine2022}. The points of $\mathcal G_d$ come from \eqref{eq:CLM-product}. Indeed, if $A_1$ maps $L^{\frac{1}{a_1}}$ to $L^{\frac{1}{b_1}}$, then $(A_1|f|^{\frac1\vartheta})^\vartheta$ belongs to $L^{\frac{1}{\vartheta b_1}}$ with norm at most a constant multiple of $\|f\|_{\frac{1}{\vartheta a_1}}$, and H\"older's inequality, which remains valid for exponents below one, takes care of the product. For a point of $\mathcal G_d$ whose third coordinate is at most one, duality converts the bound into a bound for $\Lambda$, Lemma~\ref{lem:symmetry} permutes the three exponents, and duality in the third variable converts the result back, so $\mathcal T_1$ is bounded at every point of $\mathcal C_d$. Every point of the convex hull of finitely many points can be reached by successive interpolation along segments, and Lemma~\ref{lem:positive-interpolation} therefore extends boundedness to all of $\widehat{\mathcal B}_d$.

Let
$$\mathcal O_d=\{(U_d,V_d,1-U_d),(V_d,U_d,1-U_d),(U_d,U_d,1-V_d)\}$$
be the set of the three points added to $\widehat{\mathcal B}_d$ in the definition of $\mathcal Q_d$, so that $\mathcal Q_d=\operatorname{conv}(\widehat{\mathcal B}_d\cup\mathcal O_d)$. We now show that each of these points lies outside the comparison region, so that Theorem~\ref{thm:asymmetric} enlarges the known range of exponents in every dimension $d\ge4$.

\begin{theorem}\label{thm:fixed-region}
For every $d\ge4$, every point of $\mathcal O_d$ lies outside $\widehat{\mathcal B}_d$. In particular, $\widehat{\mathcal B}_d\subsetneq\mathcal Q_d$.
\end{theorem}

\begin{proof}
Consider the affine functional $S(x,y,z)=1+x+y-z$, which is the sum of the trilinear coordinates of the point $(x,y,z)$. At each point of $\mathcal O_d$,
$$S=S_d=2U_d+V_d=\frac{3d^2-7d+6}{d^2-1}.$$
At the seven points generating $\mathcal B_d$, the functional takes the values $1$, $1$, $1$, $1+\alpha-b$, $1+\alpha-b$, $2\alpha$ and $1$, so $S\le\frac{2d}{d+1}$ on $\mathcal B_d$. At the points of $\mathcal G_d$,
$$S=1+\vartheta((a_1-b_1)+(a_2-b_2))\le1+2\vartheta\frac{d-1}{d+1},$$
because $a_i-b_i\le\alpha-b=\frac{d-1}{d+1}$ for $(a_i,b_i)\in\mathcal V_d$. Since $\vartheta\ge\frac12$ and $\frac{2d}{d+1}=1+\frac{d-1}{d+1}$, the latter bound also holds on $\mathcal B_d$. It is preserved under permutations of the trilinear coordinates, which do not change their sum, and under convex combinations, so it holds throughout $\widehat{\mathcal B}_d$. On the other hand,
$$S_d-\left(1+2\vartheta\frac{d-1}{d+1}\right)=\begin{cases}
\frac{d^2-3d+4}{d(d^2-1)}&\text{if }d\text{ is even},\\
\frac{1}{d-1}&\text{if }d\text{ is odd},
\end{cases}$$
and both quantities are positive. Hence every point of $\mathcal O_d$ lies outside $\widehat{\mathcal B}_d$, and since $\widehat{\mathcal B}_d\subset\mathcal Q_d$, the inclusion is strict.
\end{proof}

\subsection{Proofs of Theorem~\ref{thm:asymmetric} and Corollary~\ref{cor:symmetric-fixed}}\label{sec:proof-fixed}

We now interpolate between the two boundary estimates and use the symmetry of the triangle. Figure~\ref{fig:strip} shows the strip on which the interpolation takes place.

\begin{figure}[htbp]
\centering
\begin{tikzpicture}[scale=0.95]
\fill[gray!12] (-3.2,-1.9) rectangle (1.6,1.9);
\draw[->] (-4.4,0) -- (2.8,0) node[right] {$\operatorname{Re}\zeta$};
\draw[->] (0,-2.25) -- (0,2.35) node[above] {$\operatorname{Im}\zeta$};
\draw[very thick,blue!60!black] (-3.2,-1.9) -- (-3.2,1.9);
\draw[very thick,red!70!black] (1.6,-1.9) -- (1.6,1.9);
\fill (0,0) circle (2.2pt);
\node[anchor=south east] at (-0.08,0.1) {$\mathcal A^0=\mathcal T_1$};
\node[anchor=north] at (-3.2,-1.95) {$-\frac{d-3}{2}$};
\node[anchor=north] at (1.6,-1.95) {$1$};
\draw[<->,gray!60!black] (-3.2,-0.75) -- (0,-0.75);
\node[anchor=north,gray!60!black] at (-1.6,-0.78) {$\frac{d-3}{2}$};
\draw[<->,gray!60!black] (0,-0.75) -- (1.6,-0.75);
\node[anchor=north,gray!60!black] at (0.8,-0.78) {$1$};
\node[blue!60!black,align=right,anchor=east] at (-3.35,1.0) {$L^2\times L^2\to L^2$\\Proposition~\ref{prop:edgeL2}};
\node[red!70!black,align=left,anchor=west] at (1.75,1.0) {$L^{\frac{d+1}{d}}\times L^{\frac{d+1}{d}}\to L^{\frac{d+1}{2}}$\\Lemma~\ref{lem:right-boundary}};
\end{tikzpicture}
\caption{The interpolation strip $-\frac{d-3}{2}\le\operatorname{Re}\zeta\le1$, drawn to scale for $d=7$. On the left edge the family $\mathcal A^\zeta$ maps $L^2\times L^2$ to $L^2$, and on the right edge it is dominated by a product of spherical averages. The point $\zeta=0$, where $\mathcal A^0=\mathcal T_1$, lies at distance $\frac{d-3}{2}$ from the left edge and $1$ from the right edge, so it receives the weight $\lambda=\frac{2}{d-1}$ from the left edge and $1-\lambda$ from the right edge. This produces the bound $L^{\frac{d+1}{d-1}}\times L^{\frac{d+1}{d-1}}\to L^{\frac{d^2-1}{3d-5}}$ at $\zeta=0$, which the symmetry of the triangle converts into \eqref{eq:asymmetric}.}
\label{fig:strip}
\end{figure}
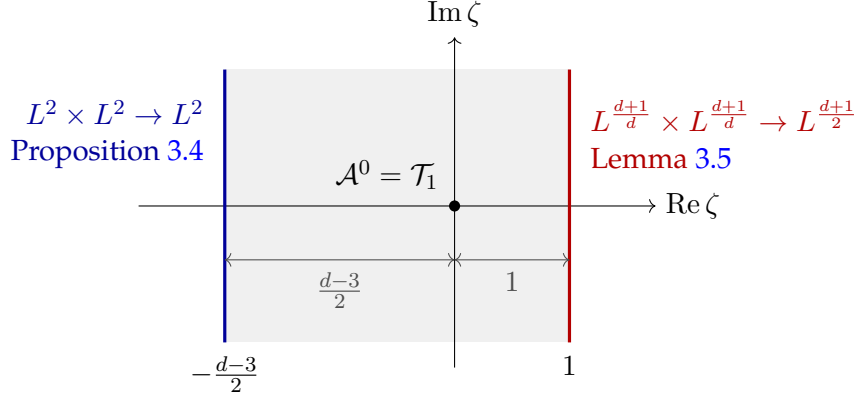

\begin{proof}[Proof of Theorem~\ref{thm:asymmetric}]
Recall that $\mathcal A^0=\mathcal T_1$. By Proposition~\ref{prop:edgeL2} with $a=\frac{d-3}{2}$ and Lemma~\ref{lem:right-boundary}, the family $\mathcal A^\zeta$ satisfies
$$\begin{aligned}
\mathcal A^{-\frac{d-3}{2}+i\tau}&:L^2\times L^2\longrightarrow L^2,\\
\mathcal A^{1+i\tau}&:L^{\frac{d+1}{d}}\times L^{\frac{d+1}{d}}\longrightarrow L^{\frac{d+1}{2}}
\end{aligned}$$
on the two edges of the strip $-\frac{d-3}{2}\le\operatorname{Re}\zeta\le1$, with constants of at most exponential growth in $|\tau|$. The point $\zeta=0$ is the convex combination
$$0=\lambda\left(-\frac{d-3}{2}\right)+(1-\lambda)\cdot1$$
with $\lambda=\frac{2}{d-1}$, so the interpolation weights at $\zeta=0$ are $\lambda$ on the left edge and $1-\lambda$ on the right edge.

We apply the bilinear analytic interpolation theorem \cite[Theorem~7.2.9]{GrafakosModern2014} to $\mathcal A^\zeta$ on this strip. That theorem is formulated for simple functions, while $\mathcal A^\zeta$ is defined on Schwartz functions, and we bridge the gap by regularization. Let $\kappa_\epsilon$ be a nonnegative, compactly supported smooth approximate identity with $\int_{\mathbb R^d}\kappa_\epsilon(x)\,dx=1$. For simple functions $f$ and $g$ with bounded support, the functions $f*\kappa_\epsilon$ and $g*\kappa_\epsilon$ belong to $C_c^\infty(\mathbb R^d)$, so $\mathcal A^\zeta_\epsilon(f,g)=\mathcal A^\zeta(f*\kappa_\epsilon,g*\kappa_\epsilon)$ is defined. By \eqref{eq:Riesz-growth}, for simple functions $f$, $g$ and $h$ with bounded support the function $\zeta\mapsto\int\mathcal A^\zeta_\epsilon(f,g)h dx$ is analytic in the open strip, continuous up to its boundary and of at most exponential growth in $|\operatorname{Im}\zeta|$, which more than suffices for the admissibility hypotheses of the theorem. Since convolution with $\kappa_\epsilon$ does not increase $L^p$ norms, the two boundary estimates hold for $\mathcal A^\zeta_\epsilon$ with constants independent of $\epsilon$. At $\zeta=0$ the interpolated exponents are
$$\begin{aligned}
(1-\lambda)\frac{d}{d+1}+\frac\lambda2&=\frac{d-1}{d+1}=U_d,\\
(1-\lambda)\frac{2}{d+1}+\frac\lambda2&=\frac{3d-5}{d^2-1}=1-V_d,
\end{aligned}$$
and we obtain $\|\mathcal T_1(f*\kappa_\epsilon,g*\kappa_\epsilon)\|_{\frac{1}{1-V_d}}\lesssim\|f\|_{\frac{1}{U_d}}\|g\|_{\frac{1}{U_d}}$ uniformly in $\epsilon$. Minkowski's inequality shows that $\|\mathcal T_1(F,G)\|_1\le\min\{\|F\|_1\|G\|_\infty,\|F\|_\infty\|G\|_1\}$, so $\mathcal T_1(f*\kappa_\epsilon,g*\kappa_\epsilon)\to\mathcal T_1(f,g)$ in $L^1$ as $\epsilon\to0$. Fatou's lemma along an almost everywhere convergent subsequence, followed by a density argument, now yields
\begin{equation}\label{eq:symmetric-edge}
\|\mathcal T_1(f,g)\|_{\frac{1}{1-V_d}}\lesssim\|f\|_{\frac{1}{U_d}}\|g\|_{\frac{1}{U_d}}
\end{equation}
for all $f,g\in L^{\frac{1}{U_d}}(\mathbb R^d)$.

To obtain the asymmetric estimate we pass to the trilinear form $\Lambda$ of Lemma~\ref{lem:symmetry}. By duality, \eqref{eq:symmetric-edge} says that
\begin{equation}\label{eq:trilinear-edge}
|\Lambda(f_1,f_2,f_3)|\lesssim\|f_1\|_{\frac{1}{U_d}}\|f_2\|_{\frac{1}{U_d}}\|f_3\|_{\frac{1}{V_d}},
\end{equation}
and since $\Lambda$ is symmetric, exchanging the roles of $f_2$ and $f_3$ leads to
$$|\Lambda(f_1,f_2,f_3)|\lesssim\|f_1\|_{\frac{1}{U_d}}\|f_2\|_{\frac{1}{V_d}}\|f_3\|_{\frac{1}{U_d}}.$$
Duality in the third variable, together with $1-U_d=\frac{2}{d+1}$, now yields
$$\|\mathcal T_1(f,g)\|_{\frac{d+1}{2}}\lesssim\|f\|_{\frac{d+1}{d-1}}\|g\|_{\frac{d^2-1}{d^2-3d+4}},$$
which is \eqref{eq:asymmetric}.

Finally, \eqref{eq:asymmetric}, the estimate obtained from it by exchanging $f$ and $g$, and \eqref{eq:symmetric-edge} are the bounds at the three points of $\mathcal O_d$. We showed in Subsection~\ref{sec:regions} that $\mathcal T_1$ is bounded at every point of $\widehat{\mathcal B}_d$, and Lemma~\ref{lem:positive-interpolation} extends boundedness to the convex hull $\mathcal Q_d$ of $\widehat{\mathcal B}_d\cup\mathcal O_d$.
\end{proof}

\begin{proof}[Proof of Corollary~\ref{cor:symmetric-fixed}]
The three points of $\mathcal O_d$ have trilinear coordinates obtained by permuting $(U_d,U_d,V_d)$, and $\mathcal T_1$ is bounded at each of them by Theorem~\ref{thm:asymmetric}. Interpolating them with equal weights by Lemma~\ref{lem:positive-interpolation} produces the point with trilinear coordinates $(c_d,c_d,c_d)$, where
$$c_d=\frac{2U_d+V_d}{3}=\frac{3d^2-7d+6}{3(d^2-1)},$$
and this point is the first estimate of the corollary. The elementary bounds at $(1,0,1)$, $(0,1,1)$ and $(0,0,0)$ have trilinear coordinates $(1,0,0)$, $(0,1,0)$ and $(0,0,1)$, and their equally weighted interpolation is the trilinear point $(\frac13,\frac13,\frac13)$, which is the bound $L^3\times L^3\to L^{\frac32}$. Since
$$c_d-\frac{d}{d+3}=\frac{2(d-3)^2}{3(d^2-1)(d+3)}>0$$
and $\frac{d}{d+3}>\frac13$, the trilinear point $(\frac{d}{d+3},\frac{d}{d+3},\frac{d}{d+3})$ lies on the segment joining $(\frac13,\frac13,\frac13)$ and $(c_d,c_d,c_d)$. Interpolating once more and returning to the operator coordinates proves the second estimate.
\end{proof}

We close this section with two comments on Corollary~\ref{cor:symmetric-fixed}. Its first estimate also lies outside $\widehat{\mathcal B}_d$ in every dimension $d\ge4$, since its trilinear coordinates have sum $3c_d=2U_d+V_d=S_d$, so the separating functional of Theorem~\ref{thm:fixed-region} applies to it as well. Its second estimate has simpler exponents, but it already follows from the comparison region when $d$ is even and $d\ge10$, or when $d$ is odd and $d\ge15$. Indeed, the point $(\vartheta\alpha,\vartheta\alpha,2\vartheta b)$ of $\mathcal G_d$ has trilinear coordinates $(\vartheta\alpha,\vartheta\alpha,1-2\vartheta b)$, and the average of its three permutations is the trilinear point $(b_d,b_d,b_d)$ with
$$3b_d=1+2\vartheta\frac{d-1}{d+1}.$$
In these dimensions $b_d\ge\frac{d}{d+3}$, and interpolation with $(\frac13,\frac13,\frac13)$ yields the second estimate without Theorem~\ref{thm:asymmetric}.

\section{\texorpdfstring{Localized maximal $L^p$-improving estimates}{Localized maximal Lp-improving estimates}}\label{sec:asymmax}

In this section we derive localized maximal estimates for $\mathcal T_E$, the maximal operator associated with a set of radii $E\subset[1,2]$, by combining fixed-radius $L^p$-improving estimates with Sobolev smoothing estimates. For each pair of frequency scales we first interpolate between these two kinds of estimates, and then we control the supremum over $E$ by the covering argument of Lemma~\ref{lem:covering}.

\subsection{From fixed-radius estimates to maximal estimates}

We begin with a transfer theorem based on the radius-covering argument of \cite{BorgesFosterOu2025}. It converts a single fixed-radius bound into a family of maximal bounds, at the price of moving the exponents towards $(\frac12,\frac12,\frac12)$ far enough for the Sobolev gain to absorb the covering number of $E$.

\begin{theorem}\label{thm:transfer}
Let $d\ge4$, and suppose that there are exponents $1<p_0,q_0<\infty$ and $1\le r_0<\infty$ and a constant $C_0>0$ such that
\begin{equation}\label{eq:transfer-input}
\|\mathcal T_1(f,g)\|_{r_0}\le C_0\|f\|_{p_0}\|g\|_{q_0}.
\end{equation}
For $0<\theta\le1$ define $p$, $q$ and $r$ by
\begin{equation}\label{eq:transfer-exponents}
\begin{aligned}
\frac1p&=\frac\theta2+\frac{1-\theta}{p_0},\\
\frac1q&=\frac\theta2+\frac{1-\theta}{q_0},\\
\frac1r&=\frac\theta2+\frac{1-\theta}{r_0}.
\end{aligned}
\end{equation}
Let $E\subset[1,2]$ be nonempty with $\overline{\dim}_{\mathrm M}E=\beta$. If $\theta>\frac{2\beta}{(d-3)r}$, then
\begin{equation}\label{eq:transfer-output}
\|\mathcal T_E(f,g)\|_r\lesssim\|f\|_p\|g\|_q.
\end{equation}
\end{theorem}

\begin{proof}
The hypothesis on $\theta$ says that $\theta\frac{d-3}{2}>\frac\beta r$. Since $r_0\ge1$, we have $\frac1r\le\frac\theta2+(1-\theta)\le1$, so $r\ge1$. We first extend the fixed-radius estimate to all radii in $[\frac12,\frac52]$ and to derivatives in the radius. For $t>0$ let $f_t(x)=f(tx)$ and $g_t(x)=g(tx)$. Since $\mathcal T_t(f,g)(tx)=\mathcal T_1(f_t,g_t)(x)$, the hypothesis \eqref{eq:transfer-input} implies
\begin{equation}\label{eq:radius-scaling}
\|\mathcal T_t(f,g)\|_{r_0}=t^{\frac{d}{r_0}}\|\mathcal T_1(f_t,g_t)\|_{r_0}\le C_0t^{\frac{d}{r_0}}\|f_t\|_{p_0}\|g_t\|_{q_0}=C_0t^{d(\frac{1}{r_0}-\frac{1}{p_0}-\frac{1}{q_0})}\|f\|_{p_0}\|g\|_{q_0},
\end{equation}
which is uniform for $t$ in compact subsets of $(0,\infty)$. We also need fixed-radius estimates for weighted operators. For a bounded measurable function $a$ on $\mathcal I$ let
$$\mathcal T_t^a(f,g)(x)=\int_{\mathcal I}a(y,z)f(x-ty)g(x-tz) d\sigma(y,z).$$
Since $\sigma$ is a positive measure, $|\mathcal T_t^a(f,g)|\le\|a\|_\infty\mathcal T_t(|f|,|g|)$, so \eqref{eq:radius-scaling} also holds for $\mathcal T^a_t$, with an additional factor $\|a\|_\infty$. This applies in particular to the weights $y_\ell$ and $z_\ell$ that arise when $\mathcal T_t$ is differentiated in $t$.

For $i,j\ge0$ let
$$B_t^{i,j}(f,g)=\mathcal T_t(P_if,P_jg),$$
where $P_if=f_i$ is the $i$-th Littlewood--Paley piece of $f$, and let $k=\max(i,j)$. The uniform estimate \eqref{eq:radius-scaling} and the uniform $L^p$ boundedness of the Littlewood--Paley operators show that $\|B_t^{i,j}(f,g)\|_{r_0}\lesssim\|f\|_{p_0}\|g\|_{q_0}$. On the other hand, the derivative formula \eqref{eq:radius-derivatives} with $t_1=t_2=t$, the weighted fixed-radius estimates and the Bernstein inequality in $L^{p_0}$ and $L^{q_0}$ lead to
$$\begin{aligned}
\|\partial_tB_t^{i,j}(f,g)\|_{r_0}&\lesssim\|\nabla P_if\|_{p_0}\|P_jg\|_{q_0}+\|P_if\|_{p_0}\|\nabla P_jg\|_{q_0}\\
&\lesssim(2^i+2^j)\|f\|_{p_0}\|g\|_{q_0}\lesssim2^k\|f\|_{p_0}\|g\|_{q_0}.
\end{aligned}$$
Together these bounds amount to
\begin{equation}\label{eq:dyadic-improving-boundary}
\|B_t^{i,j}(f,g)\|_{r_0}+2^{-k}\|\partial_tB_t^{i,j}(f,g)\|_{r_0}\lesssim\|f\|_{p_0}\|g\|_{q_0},
\end{equation}
while Lemma~\ref{lem:scale-derivative} and the inequality $\|P_if\|_2\lesssim\|f\|_2$ yield
\begin{equation}\label{eq:dyadic-smoothing-boundary}
\|B_t^{i,j}(f,g)\|_2+2^{-k}\|\partial_tB_t^{i,j}(f,g)\|_2\lesssim2^{-\frac{d-3}{2}k}\|f\|_2\|g\|_2.
\end{equation}
Both bounds hold uniformly for $t\in[\frac12,\frac52]$. Interpolating the bilinear operators $B_t^{i,j}$ and $2^{-k}\partial_tB_t^{i,j}$ separately between them, by the multilinear Riesz--Thorin theorem, we obtain
\begin{equation}\label{eq:dyadic-interpolated}
\|B_t^{i,j}(f,g)\|_r+2^{-k}\|\partial_tB_t^{i,j}(f,g)\|_r\lesssim2^{-\theta\frac{d-3}{2}k}\|f\|_p\|g\|_q
\end{equation}
uniformly for $t\in[\frac12,\frac52]$, with $p$, $q$ and $r$ given by \eqref{eq:transfer-exponents}.

We now apply Lemma~\ref{lem:covering} to $F(t,x)=B_t^{i,j}(f,g)(x)$ with $\delta=2^{-k-2}$. Since $\delta\le2^{-k}$, the estimate \eqref{eq:dyadic-interpolated} shows that the hypothesis \eqref{eq:bound_F_derivative_F} holds with a constant comparable to $2^{-\theta\frac{d-3}{2}k}\|f\|_p\|g\|_q$, and since $N(E,2^{-k-2})\le4N(E,2^{-k})$, the lemma yields
$$\left\|\sup_{t\in E}|B_t^{i,j}(f,g)|\right\|_r\lesssim N(E,2^{-k})^{\frac1r}2^{-\theta\frac{d-3}{2}k}\|f\|_p\|g\|_q.$$
By \eqref{eq:minkowski-cover}, for every $\varepsilon>0$,
\begin{equation}\label{eq:asymmax-dyadic-decay}
\left\|\sup_{t\in E}|B_t^{i,j}(f,g)|\right\|_r\lesssim_{E,\varepsilon}2^{-k(\theta\frac{d-3}{2}-\frac{\beta+\varepsilon}{r})}\|f\|_p\|g\|_q.
\end{equation}
Since $\theta\frac{d-3}{2}>\frac\beta r$, we may choose $\varepsilon>0$ so small that
$$\delta_*=\theta\frac{d-3}{2}-\frac{\beta+\varepsilon}{r}>0.$$
For Schwartz functions $f$ and $g$, the Littlewood--Paley decomposition, the triangle inequality and \eqref{eq:asymmax-dyadic-decay} imply
$$\|\mathcal T_E(f,g)\|_r\lesssim\|f\|_p\|g\|_q\sum_{i,j\ge0}2^{-\delta_*\max(i,j)}=\|f\|_p\|g\|_q\sum_{k\ge0}(2k+1)2^{-\delta_*k}\lesssim\|f\|_p\|g\|_q,$$
since there are $2k+1$ pairs $(i,j)$ with $\max(i,j)=k$, and the series converges because $\delta_*>0$.

Finally, we extend the estimate to arbitrary $f\in L^p$ and $g\in L^q$. For general inputs we define the maximal operator by taking the supremum over a countable dense subset $E'$ of $E$, which does not change $\mathcal T_E$ on Schwartz functions. The case $E=\{t\}$ of the estimate just proved shows that each $\mathcal T_t$ is bounded from $L^p\times L^q$ to $L^r$. Hence, if $f_n\to f$ in $L^p$ and $g_n\to g$ in $L^q$ are Schwartz approximations, a diagonal argument provides a subsequence along which $\mathcal T_t(f_n,g_n)\to\mathcal T_t(f,g)$ almost everywhere for every $t\in E'$. Along this subsequence $\mathcal T_{E'}(f,g)\le\liminf_n\mathcal T_{E'}(f_n,g_n)$ pointwise, and Fatou's lemma proves \eqref{eq:transfer-output}.
\end{proof}

We are now ready to prove Theorem~\ref{thm:mainasymmax}.

\begin{proof}[Proof of Theorem~\ref{thm:mainasymmax}]
Fix $\theta$ with $\theta_0<\theta\le1$, and apply Theorem~\ref{thm:transfer} with the fixed-radius exponents of Theorem~\ref{thm:asymmetric}, that is, with $\frac{1}{p_0}=U_d$, $\frac{1}{q_0}=V_d$ and $\frac{1}{r_0}=\frac{2}{d+1}$. These exponents satisfy $1<p_0,q_0<\infty$ and $2\le r_0=\frac{d+1}{2}<\infty$, and the resulting interpolated exponents are those of \eqref{eq:asymmaxexponents}. The condition of Theorem~\ref{thm:transfer} reads
$$\theta\frac{d-3}{2}>\beta\left(\frac\theta2+(1-\theta)\frac{2}{d+1}\right).$$
Multiplying by $2(d+1)$ and rearranging, we see that it is equivalent to
$$\theta((d+1)(d-3-\beta)+4\beta)>4\beta,$$
and since $(d+1)(d-3-\beta)+4\beta=(d-3)(d+1-\beta)$, the latter inequality reads $\theta>\theta_0$. The condition $d>3+\beta$ ensures that $(d-3)(d+1-\beta)>4(d-3)>4\beta$, so $\theta_0<1$, and Theorem~\ref{thm:transfer} yields \eqref{eq:mainasymmax} for every $\theta_0<\theta\le1$. Since $\sigma$ is invariant under the exchange of $y$ and $z$, the same estimate holds with $p_\theta$ and $q_\theta$ interchanged. Also, since $\sigma$ is a probability measure, $\|\mathcal T_E(f,g)\|_\infty\le\|f\|_\infty\|g\|_\infty$. Every point of $\mathcal R_{d,\beta}$ is a convex combination of finitely many of the points at which we now have bounds, and Lemma~\ref{lem:positive-interpolation}, applied to the supremum over a countable dense subset of $E$, extends boundedness to all of $\mathcal R_{d,\beta}$.
\end{proof}

\begin{remark}
The condition in Theorem~\ref{thm:transfer} is strict, since equality is not covered by the dyadic summation argument. When $\overline{\dim}_{\mathrm M}E=0$, the theorem applies for every $\theta>0$ but does not give the fixed-radius endpoint $\theta=0$. If $E$ is finite, however, the scaling estimate \eqref{eq:radius-scaling} directly shows that
$$\|\mathcal T_E(f,g)\|_{r_0}\lesssim(\#E)^{\frac{1}{r_0}}\|f\|_{p_0}\|g\|_{q_0},$$
where $\#E$ is the number of elements of $E$. For $d=4$ and $E=[1,2]$ the condition $\theta\frac{d-3}{2}>\frac{\beta}{r_\theta}$ fails for every $0<\theta\le1$, so the present argument produces no maximal estimates along \eqref{eq:asymmaxexponents} in that case, although this does not mean that such estimates fail.
\end{remark}

\subsection{The localized maximal boundedness region}\label{sec:combined}

Throughout this subsection $E\subset[1,2]$ is nonempty, $\beta=\overline{\dim}_{\mathrm M}E$ and $d>3+\beta$. We verify that $\mathcal T_E$ is bounded at every point of the four families introduced before Theorem~\ref{thm:combined-localized}, which proves that theorem, and then we compare $\mathcal C_{d,\beta}$ with the region obtained from the Sobolev and asymmetric estimates alone.

The Sobolev family is immediate. Corollary~\ref{thm:lp-improving} supplies bounds at the points $(x,y,\frac12)$ with $x,y\ge\frac12$ and $x+y<1+\frac{d-3-\beta}{2d}=\frac32-\frac{3+\beta}{2d}$, and these are the points of $\mathcal E_{d,\beta}$.

The transferred family is equally direct. Theorem~\ref{thm:asymmetric} shows that $\mathcal T_1$ is bounded at every point $(x_0,y_0,z_0)\in\mathcal Q_d$. The restrictions $0<x_0,y_0<1$ and $0<z_0\le 1$ in the definition of $\mathcal J^{\mathrm{all}}_{d,\beta}$ are the hypotheses $1<p_0,q_0<\infty$ and $1\le r_0<\infty$ of Theorem~\ref{thm:transfer}, and \eqref{eq:full-transfer-region} is its condition $\theta\frac{d-3}{2}>\frac\beta r$, so Theorem~\ref{thm:transfer} applies at every point of $\mathcal J^{\mathrm{all}}_{d,\beta}$. The three points of $\mathcal O_d$ are admissible here, because
$$V_d-\frac12=\frac{(d-3)^2}{2(d^2-1)}>0,$$
$U_d-V_d=\frac{d-3}{d^2-1}>0$ and $U_d<1$, so both input reciprocal exponents lie in $(\frac12,1)$ and the output reciprocal exponent lies in $(0,\frac12)$. For comparison, let $\mathcal J_{d,\beta}$ be the set of points $\theta(\frac12,\frac12,\frac12)+(1-\theta)(x_0,y_0,z_0)$ with $(x_0,y_0,z_0)\in\mathcal O_d$, $0<\theta\le1$ and $\theta\frac{d-3}{2}>\beta(\frac\theta2+(1-\theta)z_0)$, so that $\mathcal J_{d,\beta}\subset\mathcal J^{\mathrm{all}}_{d,\beta}$, and let
$$\mathcal M_{d,\beta}=\operatorname{conv}(\{(0,0,0)\}\cup\mathcal E_{d,\beta}\cup\mathcal J_{d,\beta})$$
be the region obtained from the Sobolev estimate and the three asymmetric points alone. It contains the region $\mathcal R_{d,\beta}$ of Theorem~\ref{thm:mainasymmax}.

For the remaining two families we need the spherical maximal bounds
\begin{equation}\label{eq:spherical-boundedness-combined}
\|M_E^{\mathrm{sph}}h\|_{L^{\frac1b}}\lesssim\|h\|_{L^{\frac1a}}
\end{equation}
for $(a,b)\in\mathscr S_{d,\beta}$, with the usual $L^\infty$ convention when $a=b=0$. Theorem~1 of Anderson, Hughes, Roos and Seeger \cite{AndersonHughesRoosSeeger2021} establishes this estimate in the interior of $\operatorname{conv}\{S_1,S_2,S_3,S_4\}$ and on the diagonal segment $\{(a,a):0\le a<\frac{d-1}{d-1+\beta}\}$. Their theorem is stated in terms of the upper Minkowski dimension $\beta$ and the Assouad dimension $\dim_{\mathrm A}E$, and since $\dim_{\mathrm A}E\le1$, their region for the pair $(\beta,\dim_{\mathrm A}E)$ contains their region for the pair $(\beta,1)$, which is the quadrilateral with vertices $S_1,\ldots,S_4$. The strong type endpoint theorem of Lee and Zhao \cite[Theorem~1.1]{LeeZhao2026} asserts that
$$\|M_{[1,2]}^{\mathrm{sph}}h\|_{L^{\frac{d^2+1}{d-1}}}\lesssim\|h\|_{L^{\frac{d^2+1}{d(d-1)}}},$$
and since $M_E^{\mathrm{sph}}h\le M_{[1,2]}^{\mathrm{sph}}h$ for $E\subset[1,2]$, the estimate at $S_4$ holds for every such $E$, regardless of $\beta$. The one-variable version of Lemma~\ref{lem:positive-interpolation}, whose proof applies verbatim to suprema of positive linear averages, now yields \eqref{eq:spherical-boundedness-combined} throughout $\mathscr S_{d,\beta}$. In particular, $S_4$ and the segment joining $S_1$ to $S_4$ are included, while the construction does not include $S_2$ or $S_3$. Roos and Seeger \cite{RoosSeeger2023} obtain a potentially larger region that uses the quasi-Assouad dimension of $E$.

The one-sided family comes from the marginals of $\sigma$. The push-forward of $\sigma$ under $(y,z)\mapsto y$ is a rotation-invariant probability measure on $S^{d-1}$, hence equal to $\sigma_{S^{d-1}}$, and the same holds for the second marginal. Consequently, $\mathcal T_E(f,g)\le\|g\|_\infty M_E^{\mathrm{sph}}(|f|)$ and $\mathcal T_E(f,g)\le\|f\|_\infty M_E^{\mathrm{sph}}(|g|)$, and together with \eqref{eq:spherical-boundedness-combined} these inequalities produce bounds at every point of $\mathcal L_{d,\beta}$.

The product family comes from \eqref{eq:CLM-product}. Dilating that estimate to radius $t$ and taking the supremum over $t\in E$, we find that
$$\mathcal T_E(f,g)(x)\lesssim(M_E^{\mathrm{sph}}(|f|^{\frac1\vartheta})(x))^\vartheta(M_E^{\mathrm{sph}}(|g|^{\frac1\vartheta})(x))^\vartheta.$$
Applying \eqref{eq:spherical-boundedness-combined} to $|f|^{\frac1\vartheta}$ and $|g|^{\frac1\vartheta}$ and using the generalized H\"older inequality, which remains valid when the output exponent is less than one, we obtain
$$\|\mathcal T_E(f,g)\|_{\frac{1}{\vartheta(b_1+b_2)}}\lesssim\|f\|_{\frac{1}{\vartheta a_1}}\|g\|_{\frac{1}{\vartheta a_2}}$$
for $(a_1,b_1),(a_2,b_2)\in\mathscr S_{d,\beta}$, with the usual convention for zero reciprocal exponents. This is the statement that $\mathcal T_E$ is bounded at every point of $\mathcal P_{d,\beta}$.

\begin{proof}[Proof of Theorem~\ref{thm:combined-localized}]
We have shown that $\mathcal T_E$ is bounded at every point of $\mathcal E_{d,\beta}$, $\mathcal J^{\mathrm{all}}_{d,\beta}$, $\mathcal L_{d,\beta}$ and $\mathcal P_{d,\beta}$, and the normalization of $\sigma$ accounts for the bound at $(0,0,0)$. Every point of $\mathcal C_{d,\beta}$ is a convex combination of finitely many points of these sets, so Lemma~\ref{lem:positive-interpolation}, applied to the supremum over a countable dense subset of $E$, extends boundedness to all of $\mathcal C_{d,\beta}$.
\end{proof}

The combined region strictly contains the region obtained from the Sobolev and asymmetric estimates, that is, $\mathcal M_{d,\beta}\subsetneq\mathcal C_{d,\beta}$. Indeed, every point of $\mathcal M_{d,\beta}$ has $z\le\frac12$, whereas the spherical maximal bounds place $(a,0,a)$ in $\mathcal C_{d,\beta}$ whenever $\frac12<a<\frac{d-1}{d-1+\beta}$, and this range is nonempty because $d-1>\beta$.

The asymmetric estimates also improve on direct interpolation of the Sobolev estimate with the $L^\infty\times L^\infty\to L^\infty$ bound. When $d=5$ and $\beta=1$, for example, the Sobolev family consists of the points $(x,y,\frac12)$ with $x,y\ge\frac12$ and $x+y<\frac{11}{10}$, so interpolation with the origin only produces points with $x+y<\frac{11z}{5}$. The point $(\frac{7}{12},\frac{13}{24},\frac{5}{12})$ supplied by Theorem~\ref{thm:mainasymmax} lies outside this region, since
$$x+y=\frac98>\frac{11}{12}=\frac{11z}{5}.$$
Figure~\ref{fig:intro-max-compare-overview} shows sections of the closure of the combined region.

\section{Necessary conditions}\label{sec:necessary-conditions}

We record necessary conditions for the fixed-radius and localized maximal estimates, writing $x=\frac1p$, $y=\frac1q$ and $z=\frac1r$. The first set of conditions is due to Borges and Foster \cite[Proposition~14]{BorgesFoster2024}.

\begin{proposition}[Fixed-radius necessary conditions]\label{prop:triangle-necessary}
Let $d\ge2$, let $1\le p,q\le\infty$ and $0<r\le\infty$, and suppose that $\|\mathcal T_1(f,g)\|_r\lesssim\|f\|_p\|g\|_q$ for all $f\in L^p(\mathbb R^d)$ and $g\in L^q(\mathbb R^d)$. Then
$$\begin{aligned}
z&\le x+y,\\
x+y&\le dz,\\
dx+y-z&\le d-1,\\
x+dy-z&\le d-1.
\end{aligned}$$
\end{proposition}

These conditions come from testing the estimate on large balls, on thin neighborhoods of the unit sphere, and on a small ball paired with such a neighborhood, the last of which is the example behind the proof of Proposition~\ref{prop:mixed-radius-set-necessary-revision} below. They also hold for $\mathcal T_E$ whenever $E$ is nonempty, since a maximal bound implies the bound for any fixed $t\in E$, which is equivalent by scaling to a bound for $\mathcal T_1$. When $E$ is large, the spread of the radii produces further restrictions.

\begin{proposition}[Radius-set necessary conditions]\label{prop:radius-set-necessary}
Let $d\ge2$, and let $E\subset[1,2]$ be nonempty with $\overline{\dim}_{\mathrm M}E=\beta$. If the estimate $\|\mathcal T_E(f,g)\|_r\lesssim\|f\|_p\|g\|_q$ holds for all $f\in L^p(\mathbb R^d)$ and $g\in L^q(\mathbb R^d)$, where $1\le p,q\le\infty$ and $0<r\le\infty$, then
$$\begin{aligned}
dx&\le d-1+(1-\beta)z,\\
dy&\le d-1+(1-\beta)z.
\end{aligned}$$
\end{proposition}

\begin{proof}
Fix $0<\delta<\frac{1}{10}$ and take $f=1_{B(0,4\delta)}$ and $g=1_{B(0,6)}$. Suppose that $||w|-t|<\delta$ for some $t\in E$, and restrict the first vertex $u$ of the triangle, which is the variable $y$ in $\sigma$, to the spherical cap of radius $c\delta$ about $\frac{w}{|w|}$, where $c>0$ is a small absolute constant. Then $|w-tu|<4\delta$, and $|w-tv|<6$ for every admissible second vertex $v$. Since the first marginal of $\sigma$ is the normalized measure on $S^{d-1}$ and the cap has measure comparable to $\delta^{d-1}$, we obtain $\mathcal T_E(f,g)(w)\gtrsim\delta^{d-1}$. Now choose a maximal $4\delta$-separated family of radii in $E$, with cardinality $M_\delta$. The corresponding annuli $\{w:||w|-t|<\delta\}$ are disjoint, their total volume is comparable to $\delta M_\delta$, and by maximality $M_\delta\ge N(E,8\delta)$. Thus, for $r<\infty$, the assumed estimate implies
$$\delta^{d-1}(\delta M_\delta)^{\frac1r}\lesssim\|f\|_p\|g\|_q\lesssim\delta^{dx}.$$
By the definition of the upper Minkowski dimension, for every $\varepsilon>0$ there are arbitrarily small $\delta$ with $M_\delta\gtrsim\delta^{-\beta+\varepsilon}$. Letting $\delta\to0$ along such a sequence and then letting $\varepsilon\to0$ yields $dx\le d-1+(1-\beta)z$. For $r=\infty$ the pointwise lower bound leads to the same conclusion with $z=0$, and interchanging $f$ and $g$ proves the other inequality.
\end{proof}

Testing with both inputs adapted to the set of radii leads to the following stronger conditions.

\begin{proposition}[Mixed necessary conditions for a set of radii]\label{prop:mixed-radius-set-necessary-revision}
Let $d\ge2$, let $E\subset[1,2]$ be nonempty, and put $\beta=\overline{\dim}_{\mathrm M}E$. Suppose that $\mathcal T_E:L^p(\mathbb R^d)\times L^q(\mathbb R^d)\to L^r(\mathbb R^d)$ is bounded, where $1\le p,q\le\infty$ and $0<r\le\infty$. Then, with $x=\frac1p$, $y=\frac1q$ and $z=\frac1r$,
\begin{equation}\label{eq:mixed-radius-set-necessary-revision}
\begin{aligned}
dx+(1-\beta)y&\le d-1+(1-\beta)z,\\
dy+(1-\beta)x&\le d-1+(1-\beta)z.
\end{aligned}
\end{equation}
\end{proposition}

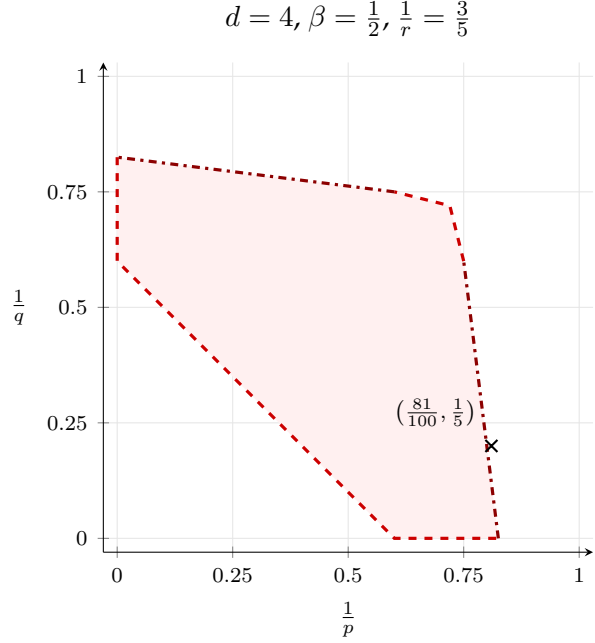
\begin{figure}[htbp]
\centering
\begin{tikzpicture}
\begin{axis}[
    width=0.52\textwidth,
    height=0.52\textwidth,
    xmin=-0.03, xmax=1.03,
    ymin=-0.03, ymax=1.03,
    axis lines=left,
    xlabel={$\frac{1}{p}$},
    ylabel={$\frac{1}{q}$},
    ylabel style={rotate=-90},
    xtick={0,0.25,0.5,0.75,1},
    ytick={0,0.25,0.5,0.75,1},
    grid=major,
    grid style={gray!20},
    tick label style={font=\scriptsize},
    label style={font=\small},
    enlargelimits=false,
    clip=true,
    title={$d=4$, $\beta=\frac{1}{2}$, $\frac{1}{r}=\frac{3}{5}$}
]

\addplot[draw=none, fill=red!6]
coordinates {(0,{3/5}) ({3/5},0) ({33/40},0) ({3/4},{3/5}) ({18/25},{18/25}) ({3/5},{3/4}) (0,{33/40})} -- cycle;

\addplot[draw=red!80!black, very thick, dashed]
coordinates {(0,{33/40}) (0,{3/5}) ({3/5},0) ({33/40},0)};

\addplot[draw=red!80!black, very thick, dashed]
coordinates {({3/4},{3/5}) ({18/25},{18/25}) ({3/5},{3/4})};

\addplot[draw=red!55!black, very thick, dashdotted]
coordinates {({33/40},0) ({3/4},{3/5})};

\addplot[draw=red!55!black, very thick, dashdotted]
coordinates {({3/5},{3/4}) (0,{33/40})};

\addplot[only marks, mark=x, mark size=3.2pt, thick, color=black]
coordinates {({81/100},{1/5})};

\node[anchor=south east, font=\scriptsize] at (axis cs:0.80,0.22) {$\left(\frac{81}{100},\frac{1}{5}\right)$};

\end{axis}
\end{tikzpicture}
\caption{The region allowed by the necessary conditions of this section for $d=4$, $\beta=\frac12$ and $\frac1r=\frac35$, drawn in the $(\frac1p,\frac1q)$-plane. The dashed edges come from Propositions~\ref{prop:triangle-necessary} and~\ref{prop:radius-set-necessary}, and the dash-dotted edges are the mixed inequalities $4x+\frac12y\le\frac{33}{10}$ and $4y+\frac12x\le\frac{33}{10}$ of Proposition~\ref{prop:mixed-radius-set-necessary-revision}, where $x=\frac1p$ and $y=\frac1q$. The cross marks a point that satisfies the conditions of the first two propositions but not the mixed ones.}
\label{fig:necessary-max-red}
\end{figure}
\begin{proof}
Fix $0<\delta<\frac{1}{100}$ and choose a maximal $20\delta$-separated finite family $\{t_1,\ldots,t_{M_\delta}\}\subset E$. By maximality $N(E,40\delta)\le M_\delta$, while the separation forces $M_\delta\le N(E,\delta)$, and consequently
$$\limsup_{\delta\downarrow0}\frac{\log M_\delta}{\log\frac{1}{\delta}}=\beta.$$
For $a>0$ let
$$G_{a\delta}=\bigcup_{j=1}^{M_\delta}\{w\in\mathbb R^d:||w|-t_j|<a\delta\}.$$
The annuli making up $G_\delta$ are disjoint, and so are those making up $G_{4\delta}$. Since $1\le t_j\le2$, their volumes satisfy $|G_\delta|\approx|G_{4\delta}|\approx\delta M_\delta$, with constants independent of $\delta$ and of the selected radii. Set $f_\delta=1_{B(0,4\delta)}$ and $g_\delta=1_{G_{4\delta}}$. If the assumed estimate is stated only for smooth inputs, we use instead nonnegative smooth majorants that equal one on these sets and are supported in $B(0,5\delta)$ and $G_{5\delta}$, respectively. Their norms obey the same upper bounds, since the $20\delta$-separation keeps the enlarged annuli disjoint.

Now let $w\in G_\delta$, choose $t_j$ with $||w|-t_j|<\delta$, and put $e=\frac{w}{|w|}$. If $u\in S^{d-1}$ satisfies $|u-e|<\frac\delta4$, then
$$|w-t_ju|\le||w|-t_j|+t_j|e-u|<\frac32\delta,$$
and for every admissible second vertex $v$, that is, every $v$ with $(u,v)\in\mathcal I$, the relation $|u-v|=1$ implies
$$||w-t_jv|-t_j|=||w-t_jv|-|t_ju-t_jv||\le|w-t_ju|<\frac32\delta.$$
Thus $f_\delta(w-t_ju)=1$ and $g_\delta(w-t_jv)=1$ for all such $u$ and $v$. The first marginal of $\sigma$ is the normalized measure on $S^{d-1}$ and the cap $\{u\in S^{d-1}:|u-e|<\frac\delta4\}$ has measure comparable to $\delta^{d-1}$, so
$$\mathcal T_E(f_\delta,g_\delta)(w)\gtrsim\delta^{d-1}$$
for every $w\in G_\delta$. The assumed estimate therefore yields, including the case $r=\infty$ with $z=0$,
$$\delta^{d-1}(\delta M_\delta)^z\lesssim\delta^{dx}(\delta M_\delta)^y.$$
Choose a sequence $\delta_n\downarrow0$ along which $\frac{\log M_{\delta_n}}{\log\frac{1}{\delta_n}}\to\beta$. Taking logarithms, dividing by $\log\frac{1}{\delta_n}$ and passing to the limit, we arrive at
$$d-1-dx+(1-\beta)(z-y)\ge0,$$
which is the first inequality in \eqref{eq:mixed-radius-set-necessary-revision}. The second one follows by interchanging the inputs. The argument includes the cases $\beta=0$ and $\beta=1$, as well as $p,q,r=\infty$ with the usual convention for reciprocal exponents.
\end{proof}

At $\beta=0$ the inequalities \eqref{eq:mixed-radius-set-necessary-revision} reduce to the conditions $dx+y-z\le d-1$ and $x+dy-z\le d-1$ of Proposition~\ref{prop:triangle-necessary}, while at $\beta=1$ they reduce to $x,y\le\frac{d-1}{d}$. Since $(1-\beta)y\ge0$, Proposition~\ref{prop:mixed-radius-set-necessary-revision} also implies Proposition~\ref{prop:radius-set-necessary}, and the improvement is visible in Figure~\ref{fig:necessary-max-red}. For the parameters of that figure, the point $(x,y)=(\frac{81}{100},\frac15)$ satisfies the conditions of Propositions~\ref{prop:triangle-necessary} and~\ref{prop:radius-set-necessary}, but
$$4x+\frac12y=\frac{167}{50}>\frac{33}{10}=3+\frac12\cdot\frac35,$$
so it is excluded by the first mixed inequality. We do not know whether the necessary conditions of Propositions~\ref{prop:radius-set-necessary} and~\ref{prop:mixed-radius-set-necessary-revision} are sharp.

\end{document}